\documentclass[12pt]{amsart}

\usepackage{amsfonts}
\usepackage{amsmath}
\usepackage{amssymb}
\usepackage{latexsym}
\usepackage{amscd}
\usepackage{amsthm}
\usepackage[dvips]{graphicx}
\usepackage{subfigure}
\usepackage[all]{xy}
\usepackage{color}

\newcommand{\I}{\mathcal{I}}
\newcommand{\M}{\mathcal{M}}
\newcommand{\Z}{\mathbb{Z}}
\newcommand{\AR}{{\rm Aut}(H_1(N_g;R),\cdot)}
\newcommand{\AZ}{{\rm Aut}(H_1(N_g;\Z),\cdot)}
\newcommand{\Az}{{\rm Aut}(H_1(N_g;\Z/{2\Z}),\cdot)}
\newcommand{\la}{\langle}
\newcommand{\ra}{\rangle}
\newcommand{\g}{\Gamma_2}

\newtheorem{thm}{Theorem}[section]
\newtheorem{prop}[thm]{Proposition}
\newtheorem{lem}[thm]{Lemma}
\newtheorem{cor}[thm]{Corollary}
\theoremstyle{example}
\newtheorem{exam}[thm]{Example}
\newtheorem{rem}[thm]{Remark}

\author[S. Hirose]{Susumu Hirose}
\address{Department of Mathematics, Faculty of Science and Technology, Tokyo University of Science, Noda, Chiba, 278-8510, Japan}
\email{hirose\_susumu@ma.noda.tus.ac.jp}
\thanks{The first author was supported by Grant-in-Aid for Scientific
Research (C) (No. 24540096), Japan Society for the Promotion of Science}
\author[R. Kobayashi]{Ryoma Kobayashi}
\address{Department of General Education, Ishikawa National College of
Technology, Tsubata, Ishikawa, 929-0392, Japan}
\email{kobayashi\_ryoma@ishikawa-nct.ac.jp}

\begin{document}

\title[A normal generating set for $\I(N_g)$]
{A normal generating set for the Torelli group of a
non-orientable closed surface}
\maketitle

\begin{abstract}
For a closed surface $S$, its Torelli group $\I(S)$ is the subgroup of
 the mapping class group of $S$ consisting of elements acting trivially
 on $H_1(S;\Z)$.
When $S$ is orientable, a generating set for $\I(S)$ is known (see
 \cite{p}).
In this paper, we give a normal generating set of $\I(N_g)$ for
 $g\geq4$, where $N_g$ is a genus-$g$ non-orientable closed surface.
\end{abstract}

\section{Introduction}\label{intro}

For a closed connected {\it non-orientable} surface $S$, the {\it
mapping class group} $\M(S)$ of $S$ is defined to be the group of
isotopy classes of all diffeomorphisms over $S$.
For a closed connected {\it orientable} surface $S$, the {\it mapping
class group} $\M(S)$ of $S$ is defined to be the group of isotopy
classes of all {\it orientation-preserving} diffeomorphisms over $S$.
In this paper, for $x,y\in\M(S)$ the composition $yx$ means that we
first apply $x$ and then $y$.
The {\it Torelli group} $\I(S)$ of $S$ is the subgroup of $\M(S)$
consisting of elements acting trivially on $H_1(S;\Z)$.
Let $\Sigma_g$ be a genus-$g$ orientable closed surface.
Powell~\cite{p} showed that $\I(\Sigma_g)$ is generated by 
{\it BSCC maps} and {\it BP maps}.
In \cite{pu}, Putman proved Powell's result more conceptually.
In addition, Johnson~\cite{j} showed that $\I(\Sigma_g)$ is generated by
a finite number of BP maps.
In this paper, we consider the case where $S$ is a non-orientable closed
surface.

Let $N_g$ denote a genus-$g$ non-orientable closed surface, that is,
$N_g$ is a connected sum of $g$ real projective planes.
As another classification, we see that $N_g$ is a connected sum of a
genus-$h$ orientable closed surface with $(g-2h)$ real projective
planes, for $0\leq{h}<\frac{g}{2}$.
In this paper, we regard $N_g$ as a surface which is obtained by
attaching $g-2h$ M\"obius bands to a genus-$h$ compact orientable
surface with $g-2h$ boundaries for $0\leq{h}<\frac{g}{2}$ (see
Figure~\ref{non-ori}).
For $R=\Z$ and $\Z/{2\Z}$, let
$\cdot:H_1(N_g;R)\times{H_1(N_g;R)}\to\Z/{2\Z}$ be the mod $2$
intersection form, and let $\AR$ be the group of automorphisms over
$H_1(N_g;R)$ preserving the mod $2$ intersection form.
McCarthy-Pinkall~\cite{mp} and Gadgil-Pancholi~\cite{gp} proved that the
natural homomorphism $\rho:\M(N_g)\to\AZ$ is surjective.

\begin{figure}[htbp]
\includegraphics[scale=0.5]{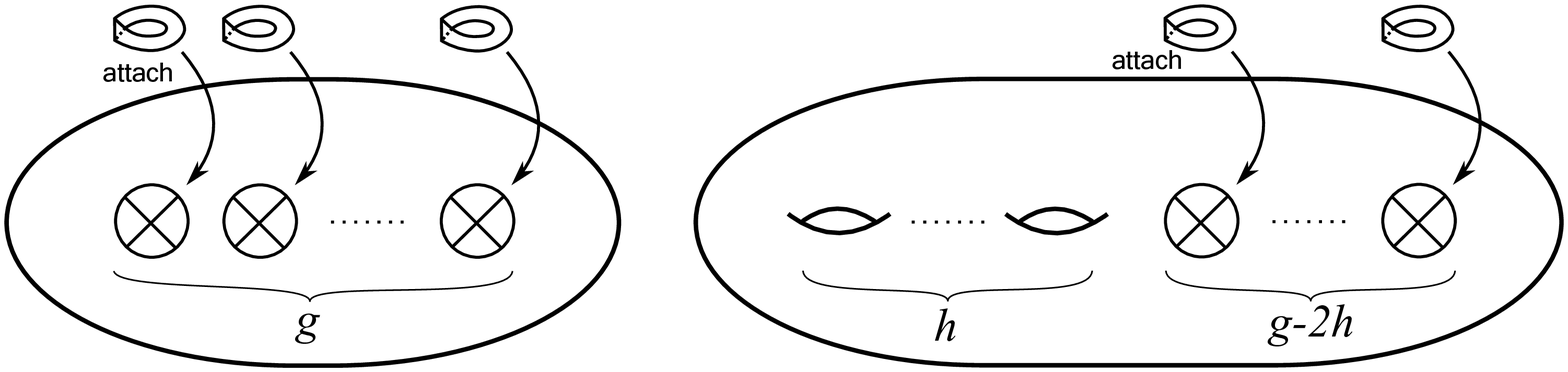}
\caption{A genus-$g$ non-orientable closed surface $N_g$.}\label{non-ori}
\end{figure}

Lickorish~\cite{l1} showed that $\M(N_g)$ is generated by Dehn twists
and $Y$-homeomorphisms.
In addition, Lickorish~\cite{l2} showed that the subgroup of $\M(N_g)$
generated by Dehn twists is an index $2$ subgroup of $\M(N_g)$.
Hence $\M(N_g)$ is not generated by Dehn twists.
On the other hand, since a $Y$-homeomorphism acts on $H_1(N_g;\Z/{2\Z})$
trivially, $\M(N_g)$ is not generated by $Y$-homeomorphisms. 
Chillingworth~\cite{c} found a finite generating set for $\M(N_g)$.
Presentations for $\M(N_1)$ and $\M(N_2)$ are known classically.
A finite presentation for $\M(N_3)$ is obtained by Birman and
Chillingworth in \cite{bc}.
A finite presentation for $\M(N_4)$ is obtained by Szepietowski in
\cite{s1}.
Finally, a finite presentation for $\M(N_g)$ is obtained by Paris,
Szepietowski~\cite{ps} and Stukow~\cite{s} for $g\geq4$.

For a simple closed curve $c$ on $N_g$, $c$ is called an {\it $A$-circle}
(resp. an {\it $M$-circle}) if its regular neighborhood is an annulus
(resp. a M\"obius band) (see Figure~\ref{am-circle}).
Let $a$ and $m$ be an $A$-circle and an $M$-circle on $N_g$ respectively.
Suppose that $a$ and $m$ intersect transversely at only one point.
We define a $Y$-homeomorphism $Y_{m,a}$ as follows.
Let $K$ be a regular neighborhood of $a\cup m$ in $N_g$, and let $M$ be
a regular neighborhood of $m$ in the interior of $K$.
Note that $K$ is homeomorphic to the Klein bottle with a boundary.
$Y_{m,a}$ is a homeomorphism over $N_g$ which is described as the result
of pushing $M$ once along $a$ keeping the boundary of $K$ fixed 
(see Figure~\ref{y-homeo}).
For an $A$-circle $c$ on $N_g$, we denote by $t_c$ a Dehn twist about
$c$, and the direction of the twist is indicated by a small arrow
written beside $c$ as shown in Figure~\ref{dehn}.

\begin{figure}[htbp]
\subfigure[$A$-circles]{\includegraphics[scale=0.5]{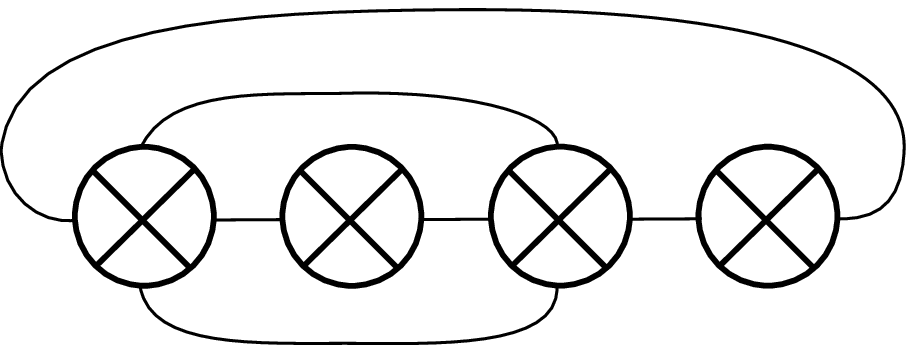}}\hspace{0.5cm}
\subfigure[$M$-circles]{\includegraphics[scale=0.5]{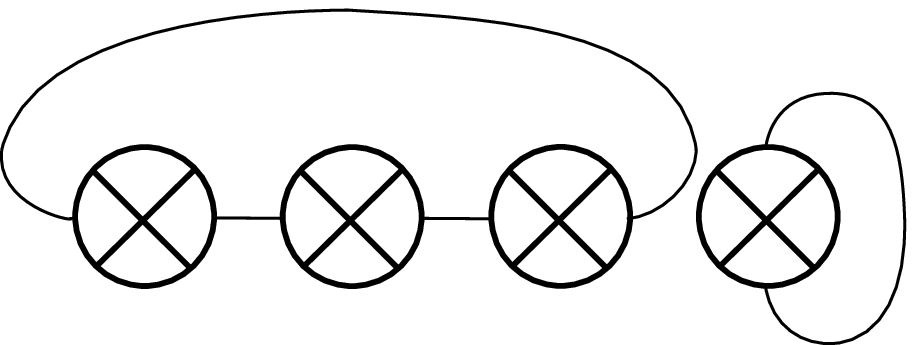}}
\caption{}\label{am-circle}
\end{figure}

\begin{figure}[htbp]
\includegraphics[scale=0.5]{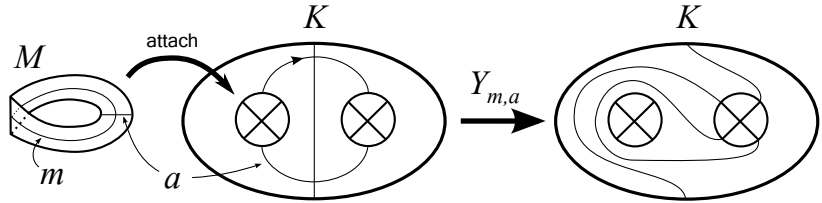}
\caption{A $Y$-homeomorphism $Y_{m,a}$.}\label{y-homeo}
\end{figure}

\begin{figure}[htbp]
\includegraphics{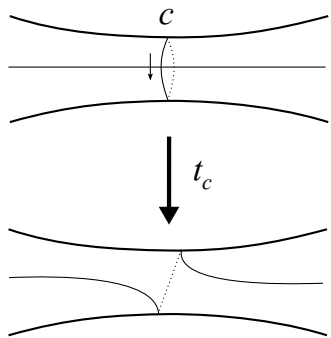}
\caption{A Dehn twist $t_c$ about $c$.}\label{dehn}
\end{figure}

Let $c$ be an $A$-circle on $N_g$ such that $N_g\setminus{c}$ is not
connected.
We call $t_c$ a \textit{bounding simple closed curve map}, for short a
\textit{BSCC map} (see Figure~\ref{BSCCBP}~(a)).
Let $c_1$ and $c_2$ be $A$-circles on $N_g$ such that $N_g\setminus{c_i}$
is connected, $N_g\setminus(c_1\cup c_2)$ is not connected and one of
its connected components is orientable.
We call $t_{c_1}t_{c_2}^{-1}$ a \textit{bounding pair map}, for short a
\textit{BP map} (see Figure~\ref{BSCCBP}~(b)).
In Section~\ref{basic}, we will see that BSCC maps and BP maps are in
$\I(N_g)$.

\begin{figure}[htbp]
\subfigure[a BSCC map $t_c$.]{\includegraphics[scale=0.5]{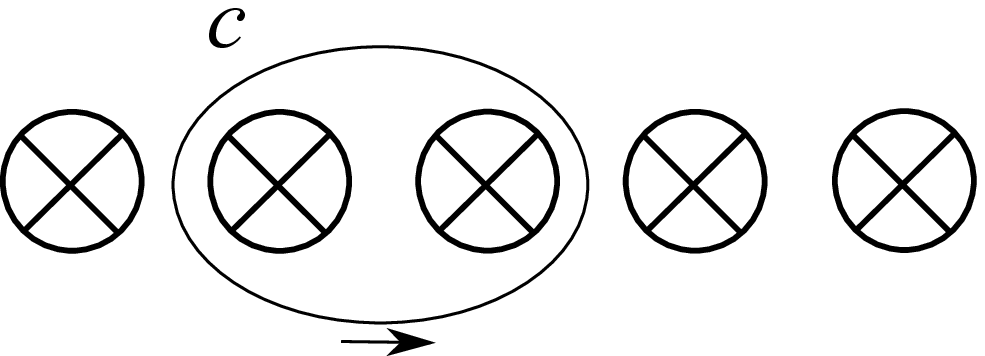}}

\subfigure[a BP map $t_{c_1}t_{c_2}^{-1}$.]{\includegraphics[scale=0.5]{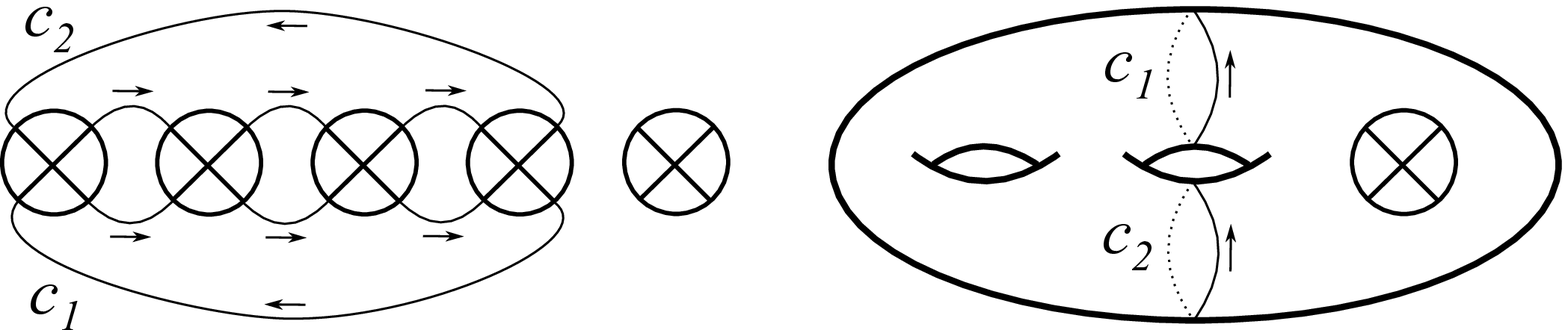}}
\caption{}\label{BSCCBP}
\end{figure}

For $h,b\geq1$, let $N_h^b$ be a non-orientable surface of genus $h$
with $b$ boundary components, and let $\Sigma_h^b$ an orientable surface
of genus $h$ with $b$ boundary components.
Our main result is the following.

\begin{thm}\label{thm}
For $g\geq5$, $\I(N_g)$ is generated by following elements.
\begin{itemize}
 \item BSCC maps $t_c$ such that one of connected components of
       $N_g\setminus{c}$ is homeomorphic to $N_2^1$, the other component
       of $N_g\setminus{c}$ is non-orientable.
 \item BP maps $t_{c_1}t_{c_2}^{-1}$ such that one of connected
       components of $N_g\setminus(c_1\cup c_2)$ is homeomorphic to
       $\Sigma_1^2$, the other component of $N_g\setminus(c_1\cup c_2)$
       is non-orientable.
\end{itemize}
$\I(N_4)$ is generated by following elements.
\begin{itemize}
 \item BSCC maps $t_c$ such that one of connected components of
       $N_4\setminus{c}$ is homeomorphic to $N_2^1$, the other component
       of $N_4\setminus{c}$ is non-orientable.
 \item BSCC maps $t_c$ such that one of connected components of
       $N_4\setminus{c}$ is homeomorphic to $N_2^1$, the other component
       of $N_4\setminus{c}$ is orientable.
 \item BP maps $t_{c_1}t_{c_2}^{-1}$ such that one of connected
       components of $N_4\setminus(c_1\cup c_2)$ is homeomorphic to
       $\Sigma_1^2$, the other component of $N_4\setminus(c_1\cup c_2)$
       is an annulus as shown in Figure~\ref{t-a}.
       \begin{figure}[htbp]
	\includegraphics[scale=0.5]{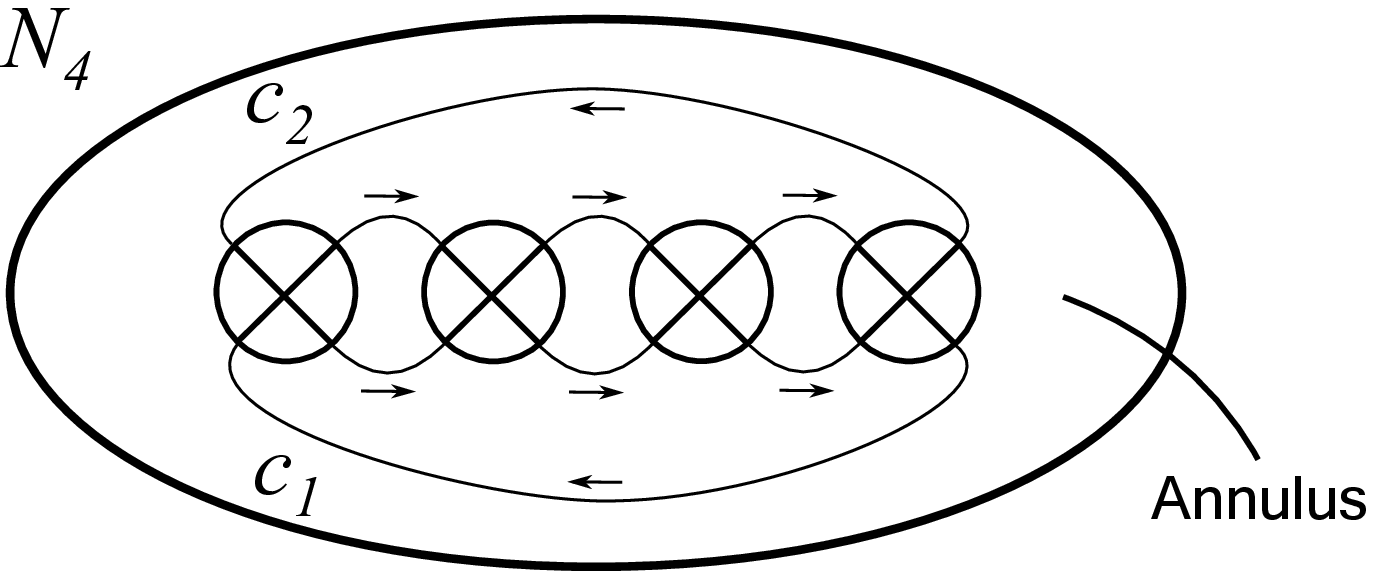}
	\caption{}\label{t-a}
       \end{figure}
\end{itemize}
\end{thm}

In this theorem, these generating sets are infinite.
We do not know that whether or not $\I(N_g)$ can be finitely generated,
and generated by only BSCC maps or BP maps.

Here is an outline of a proof of Theorem~\ref{thm}.
Let $\g(N_g)$ be the subgroup of $\M(N_g)$ consisting of elements acting
trivially on $H_1(N_g;\Z/{2\Z})$.
We call $\g(N_g)$ the {\it level $2$ mapping class group} of $N_g$.
Note that $\I(N_g)\subset\g(N_g)$.
Let $\Phi_g:\AZ\to\Az$ be the natural epimorphism.
Consider the natural homomorphism $\rho':\g(N_g)\to\ker\Phi_g$.
Then we have that $\I(N_g)$ is equal to $\ker\rho'$.
Let $\g(n)=\ker(GL(n;\Z)\to{GL(n;\Z/{2\Z})})$.
We call $\g(n)$ the {\it level $2$ principal congruence subgroup} of
$GL(n;\Z)$.
McCarthy-Pinkall~\cite{mp} showed that $\ker\Phi_g$ is isomorphic to
$\g(g-1)$.
On the other hand, Szepietowski \cite{s3} gave a finite generating set
for $\g(N_g)$, and then the first author and Sato \cite{hs} gave a
minimal generating set for $\g(N_g)$.
Fullarton \cite{f}, the second author \cite{k}, Margalit and Putman gave
a finite presentation for $\g(n)$ independently.
Therefore, we obtain a normal generating set for $\I(N_g)$ in
$\g(N_g)$.

Here is an outline of this paper.
In Section~\ref{basic}, we will explain about basics on the Torelli
group of a non-orientable surface.
In Section~\ref{preliminary}, we will explain about the finite
generating set for $\g(N_g)$, the finite presentation for $\g(n)$ and an
isomorphism from $\ker\Phi_g$ to $\g(g-1)$.
In Section~\ref{norgen}, we will obtain a normal generating set for
$\I(N_g)$.
In Section~\ref{proofthm}, we will show that each normal generator of
$\I(N_g)$ obtained in Section~\ref{norgen} is described as a product of
BSCC maps and BP maps.

\section{Basics on the Torelli group of a non-orientable surface.}\label{basic}

There are BSCC maps of two types.
A BSCC map $t_c$ is called a BSCC map of type $(1,h)$ if each connected
component of $N_g\setminus{c}$ is non-orientable and one component of
$N_g\setminus{c}$ is homeomorphic to $N_h^1$ for
$1\leq{h}\leq\frac{g}{2}$ (see Figure~\ref{BSCCBP}~(a)).
A BSCC map $t_c$ is called a BSCC map of type $(2,h)$ if one component
of $N_g\setminus{c}$ is homeomorphic to $\Sigma_h^1$ for
$1\leq{h}<\frac{g}{2}$, and the other component is non-orientable (see
Figure~\ref{type2}).
Note that a BSCC map $t_c$ is trivial if $c$ bounds a M\"obius band (see
Theorem~3.4 of \cite{e}).

There are BP maps of two types.
A BP map $t_{c_1}t_{c_2}^{-1}$ is called a BP map of type $(1,h)$ if one
component of $N_g\setminus(c_1\cup{c_2})$ is homeomorphic to
$\Sigma_h^2$ for $1\leq{h}<\frac{g}{2}-1$, and the other component is
non-orientable (see Figure~\ref{BSCCBP}~(b)).
A BP map $t_{c_1}t_{c_2}^{-1}$ is called a BP map of type $(2,h)$ if
each component of $N_g\setminus(c_1\cup{c_2})$ is orientable and one
component of $N_g\setminus(c_1\cup{c_2})$ is homeomorphic to
$\Sigma_h^2$ for $1\leq{h}\leq\frac{g}{2}-1$ (see Figure~\ref{type2}).
Note that a BP map of type $(2,h)$ appears only if $g$ is even.

\begin{figure}[htbp]
\includegraphics[scale=0.5]{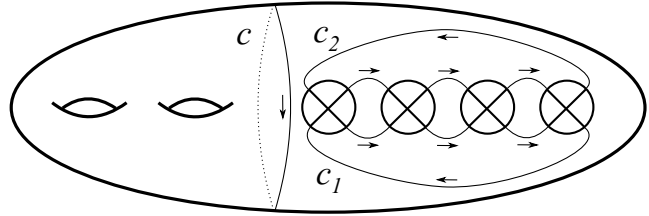}
\caption{A BSCC map $t_{c}$ of type $(2,2)$ and a BP map
 $t_{c_1}t_{c_2}^{-1}$ of type $(2,1)$.}\label{type2}
\end{figure}

At first, we show the following.

\begin{rem}\label{bsccbprem}
All BSCC maps and BP maps are in $\I(N_g)$.
\end{rem}

\begin{proof}\label{bsccbpremproof}
Let $c$, $d$, $c_1$, $c_2$, $d_1$ and $d_2$ be simple closed curves on
 $N_g$ as shown in Figure~\ref{bsccbp}.
Note that $t_c$ is a BSCC map of type $(1,h)$, $t_d$ is a BSCC map of
 type $(2,h)$, $t_{c_1}t_{c_2}^{-1}$ is a BP map of type $(1,h)$ and
 $t_{d_1}t_{d_2}^{-1}$ is a BP map of type $(2,h)$.
In $\M(N_g)$, any BSCC map of type $(1,h)$ (resp. type $(2,h)$) is
 conjugate to $t_c^{\pm1}$ (resp. $t_d^{\pm1}$), and any BP map of type
 $(1,h)$ (resp. type $(2,h)$) is conjugate to
 $(t_{c_1}t_{c_2}^{-1})^{\pm1}$ (resp. $(t_{d_1}t_{d_2}^{-1})^{\pm1}$).
Hence it suffice to show that $t_c$, $t_d$, $t_{c_1}t_{c_2}^{-1}$ and
 $t_{d_1}t_{d_2}^{-1}$ are in $\I(N_g)$.

\begin{figure}[htbp]
\subfigure[A BSCC map $t_c$ of type $(1,h)$.]{\includegraphics[scale=0.5]{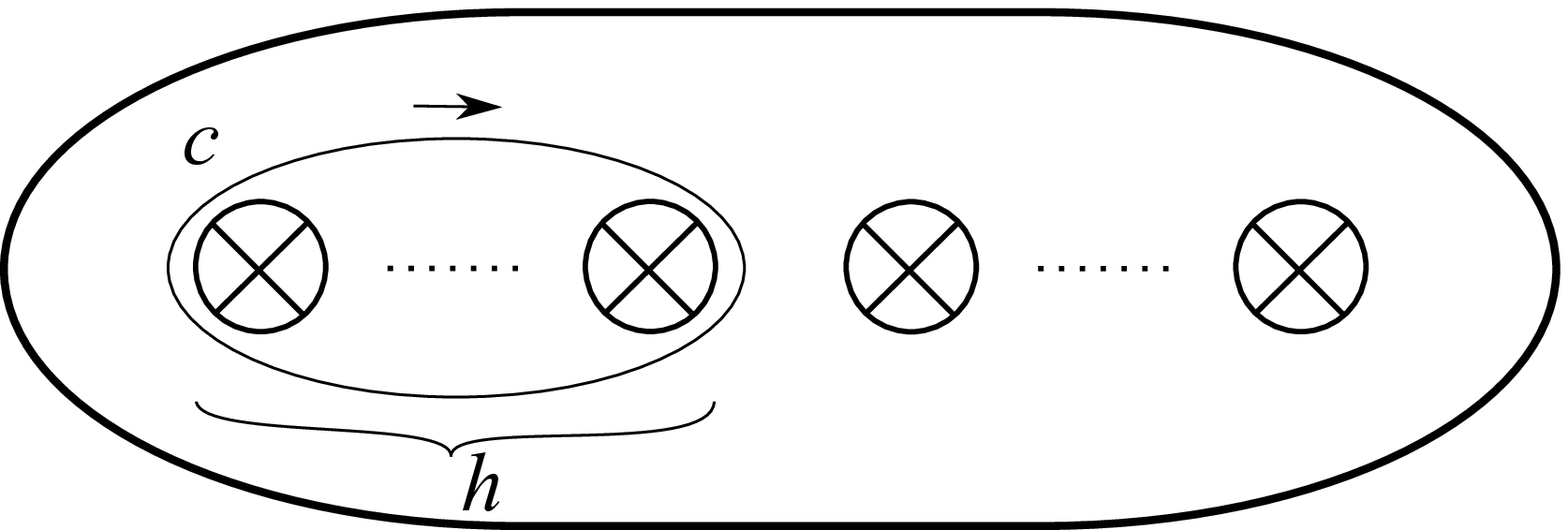}}
\subfigure[A BSCC map $t_d$ of type $(2,h)$.]{\includegraphics[scale=0.5]{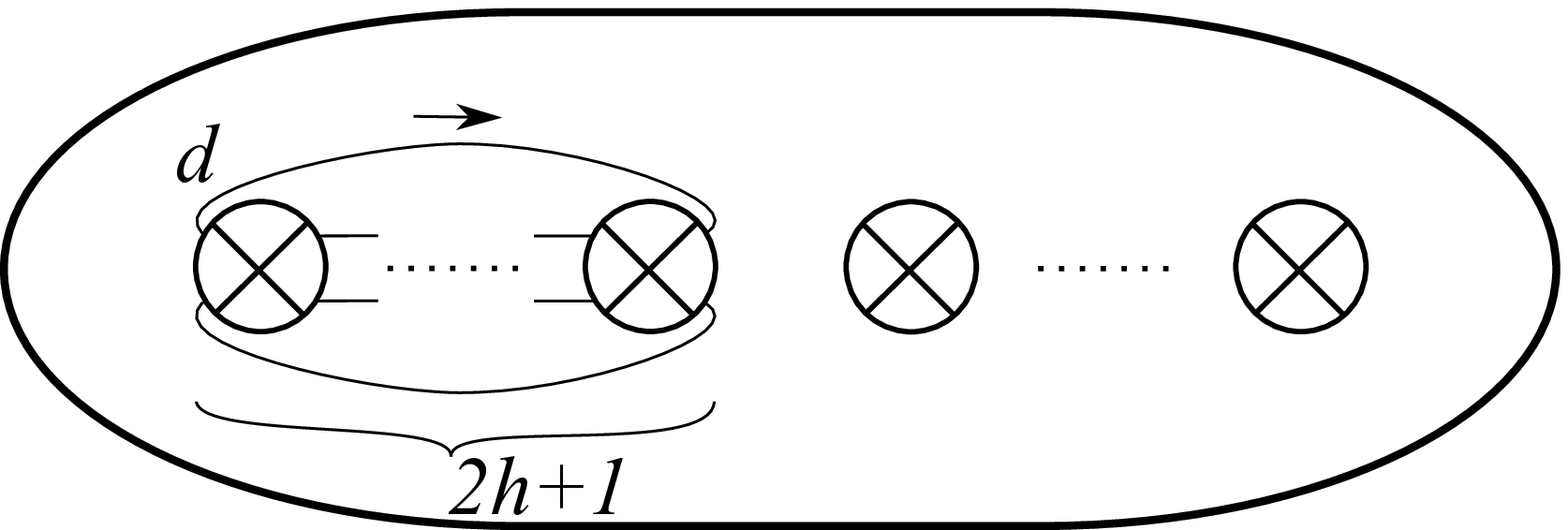}}
\subfigure[A BP map $t_{c_1}t_{c_2}^{-1}$ of type $(1,h)$.]{\includegraphics[scale=0.5]{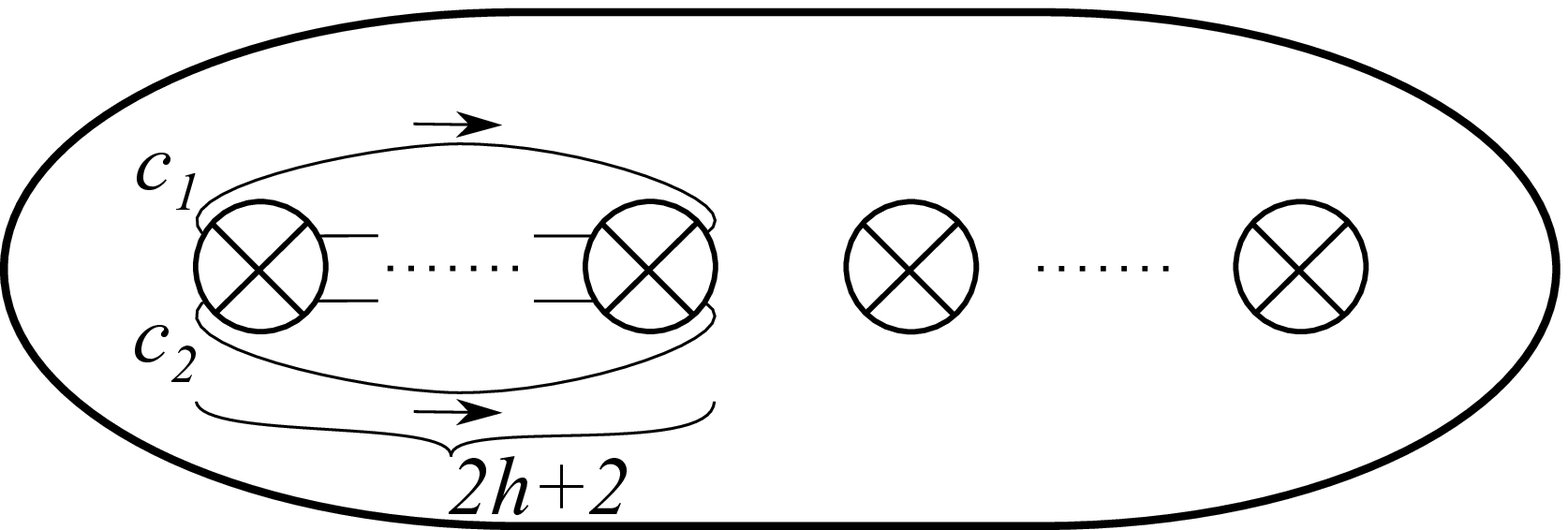}}
\subfigure[A BP map $t_{d_1}t_{d_2}^{-1}$ of type $(2,h)$.]{\includegraphics[scale=0.5]{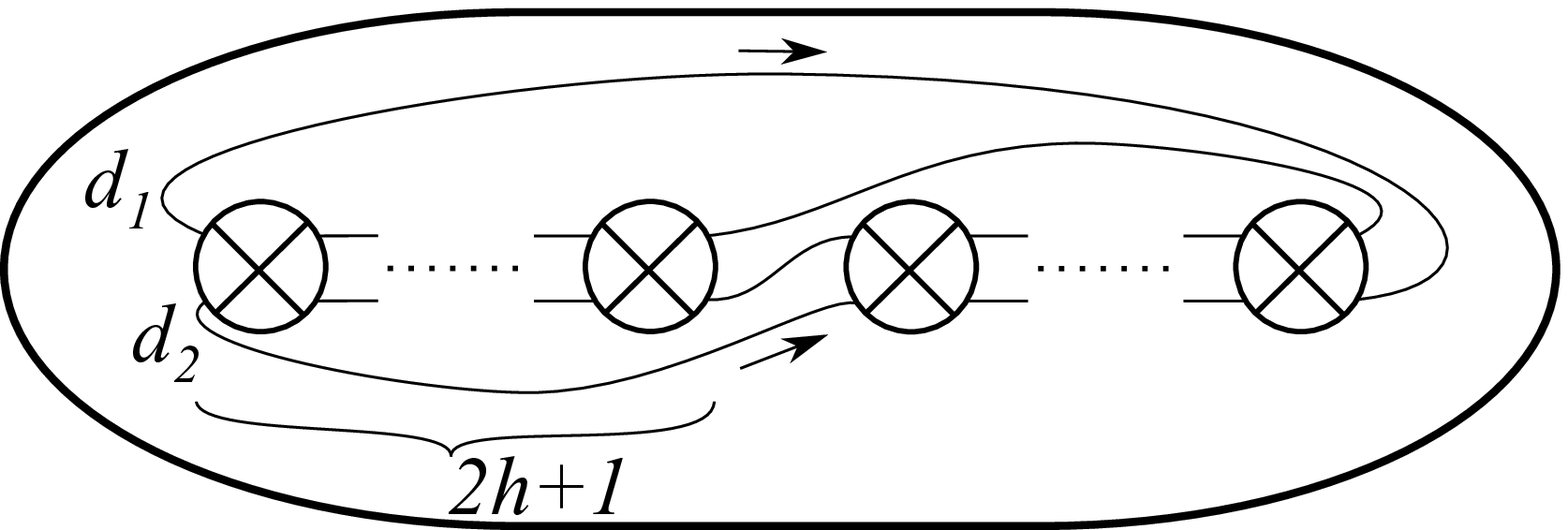}}
\caption{}\label{bsccbp}
\end{figure}

For $1\leq{i}\leq{g}$, let $\alpha_i$ be a simple closed curve on $N_g$
 as shown in Figure~\ref{alpha_i}, and let
 $c_i=[\alpha_i]\in{}H_1(N_g;\Z)$.
By a natural handle decomposition whose cores of the $1$-handles are
 $\alpha_i$, we have that $H_1(N_g;\Z)$ is generated by $c_i$, as a
 $\Z$-module (see Figure~\ref{handle}).
We can see that $t_c$, $t_d$, $t_{c_1}t_{c_2}^{-1}$ and
 $t_{d_1}t_{d_2}^{-1}$ act trivially on each $c_i$.

\begin{figure}[htbp]
\includegraphics[scale=0.5]{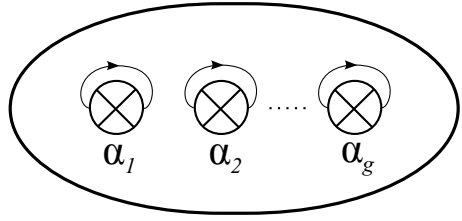}
\caption{The loops $\alpha_1,\alpha_2,\dots,\alpha_g$.}\label{alpha_i}
\end{figure}

\end{proof}

Next we prove the following.

\begin{lem}\label{lem}
For $g\geq5$, we have followings.
\begin{enumerate}
 \item \begin{enumerate}
	\item Any BSCC map of type $(2,\frac{g}{2}-1)$ in the case where
	      $g$ is even is a product of BP maps of type $(1,1)$.
	\item Any BSCC map of the other types is a product of BSCC maps
	      of type $(1,2)$.
       \end{enumerate}
 \item Any BP map of type $(1,h)$ is a product of BP maps of type
       $(1,1)$.
       Any BP map of type $(2,h)$ is a product of BP maps of type
       $(2,1)$.
\end{enumerate}
\end{lem}

\begin{proof}
In a proof, we use ideas of Johnson \cite{j}.
\begin{enumerate}
 \item \begin{enumerate}
	\item We first show that a BSCC map of type
	      $(2,\frac{g}{2}-1)$ is a product of BP maps.
	      Let $t_c$ be a BSCC map of type $(2,\frac{g}{2}-1)$.
	      Then the curve $c$ is as shown in Figure~\ref{bscc}~(a).
	      Let $x$, $y$, $z$, $a$, $b$ and $d$ be simple closed
	      curves as shown in Figure~\ref{bscc}~(a).
	      By the lantern relation, we have the relation
	      $t_dt_ct_bt_a=t_zt_yt_x$.
	      Since $a$, $b$, $c$ and $d$ are not intersect other loops,
	      we have $t_c=(t_zt_a^{-1})(t_yt_d^{-1})(t_xt_b^{-1})$.
	      Note that $t_xt_b^{-1}$, $t_yt_d^{-1}$ and $t_zt_a^{-1}$
	      are BP maps.
	      Hence a BSCC map of type $(2,\frac{g}{2}-1)$ is a product
	      of BP maps.
	      As we will show in the assertion (2), these BP maps are
	      products of BP maps of type $(1,1)$.
	      Hence we obtain the claim.
	\item Let $t_c$ be a BSCC map of type $(1,h)$ or
	      $(2,\frac{g-h}{2})$ for $h\geq3$, then $c$ is as shown in
	      Figure~\ref{bscc}~(b).
	      let $c_{i,j}$ be a simple closed curve for
	      $1\leq{i<j}\leq{h}$ as shown in Figure~\ref{bscc}~(b).
	      We have
	      \begin{equation}\label{append(I)}
		t_c=\prod_{1\leq{i}\leq{h-1}}(t_{c_{i,i+1}}t_{c_{i,i+2}}\cdots{t_{c_{i,h-1}}t_{c_{i,h}}}).\tag{I}
	      \end{equation}
	      The equation (\ref{append(I)}) will be shown in
	      Appendix~\ref{appendix}.
	      Since each $t_{c_{i,j}}$ is a BSCC map of type $(1,1)$, we
	      obtain the claim.
	      \begin{figure}[htbp]
	       \subfigure[Loops $a,b,c,d$ and $x,y,z$.]{\includegraphics[scale=0.5]{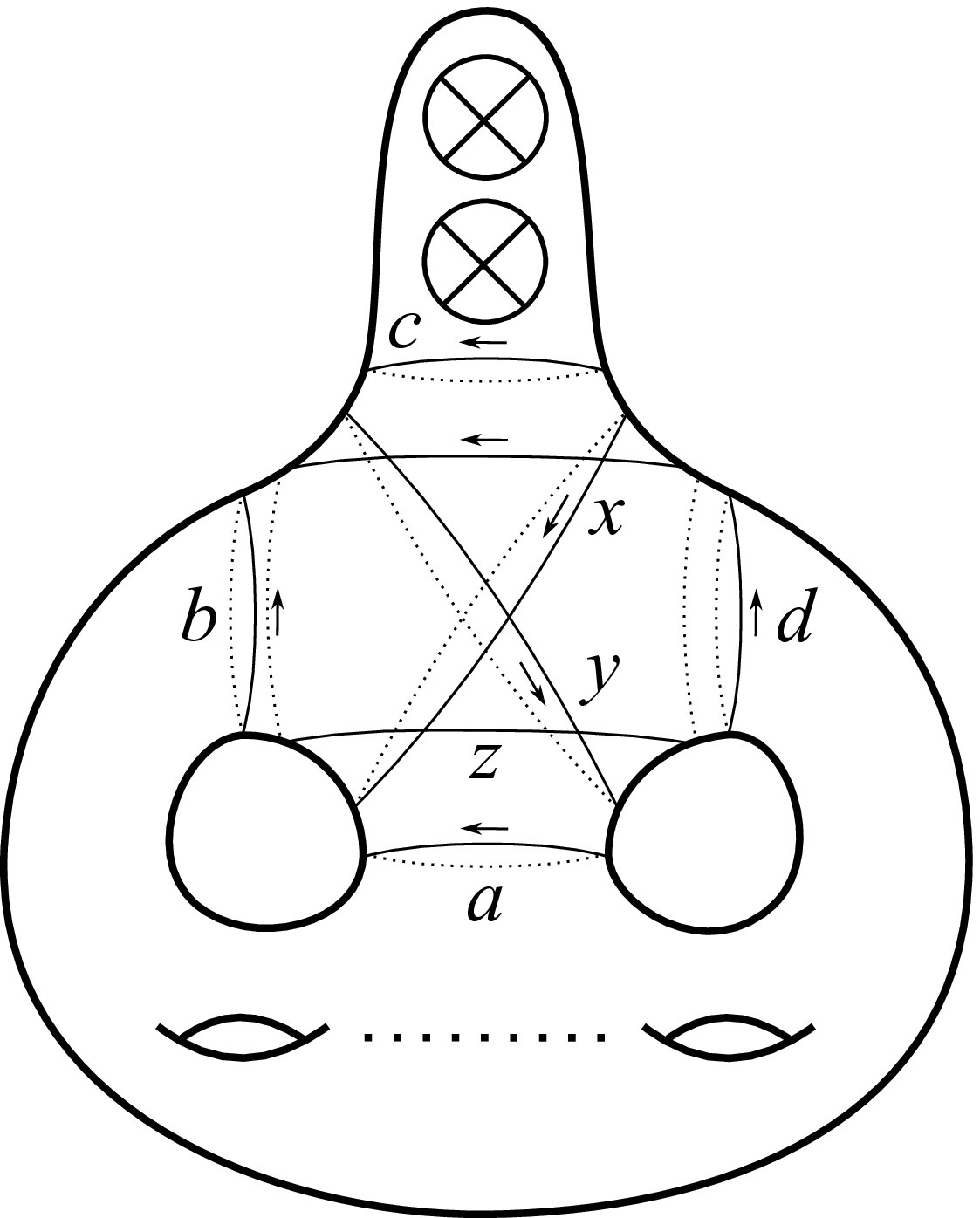}}\hspace{0.5cm}
	       \subfigure[Loops $c$ and $c_{i,j}$.]{\includegraphics[scale=0.5]{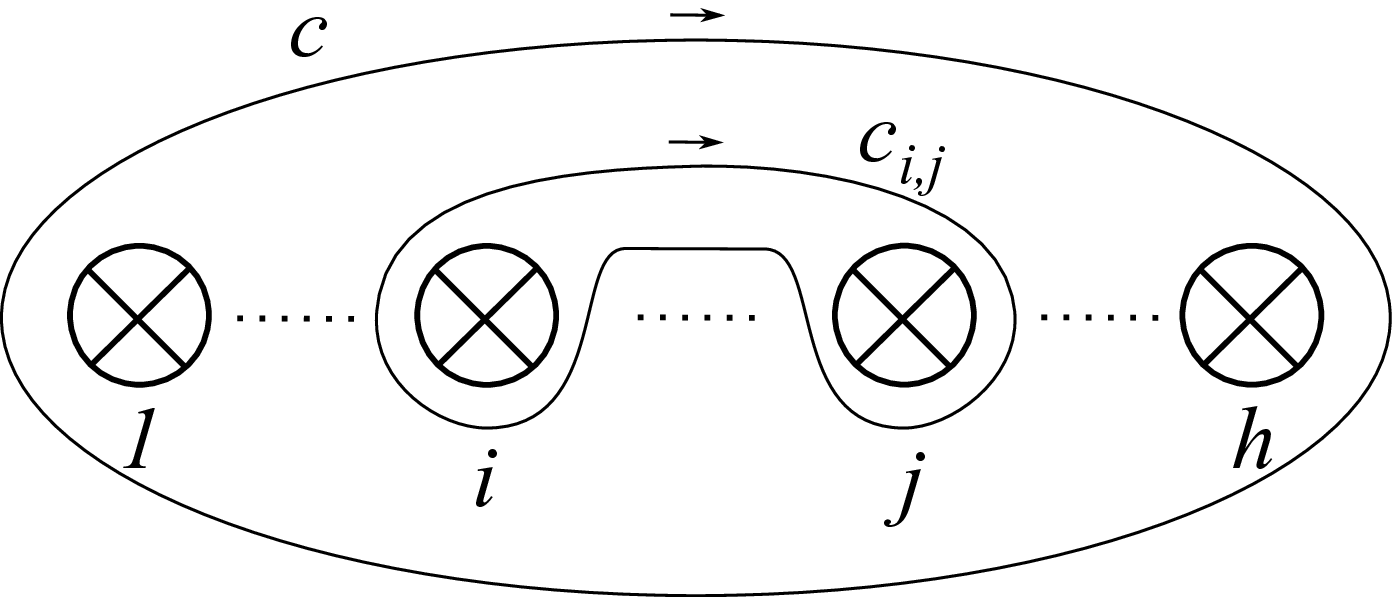}}
	       \caption{}\label{bscc}
	      \end{figure}
       \end{enumerate}
 \item For $h\geq1$, let $d_0,d_1,\dots,d_h$ be simple closed curves as
       shown in Figure~\ref{bp}.
       Suppose that $d_0$ and $d_h$ are not separating curves.
       Note that $t_{d_0}t_{d_h}^{-1}$ is a BP map of type $(1,h)$ or
       $(2,h)$.
       Then we have the equation
       $t_{d_0}t_{d_h}^{-1}=(t_{d_0}t_{d_1}^{-1})(t_{d_1}t_{d_2}^{-1})\cdots(t_{d_{h-1}}t_{d_h}^{-1})$.
       Since each $t_{d_i}t_{d_{i+1}}^{-1}$ is a BP map of type $(1,1)$
       or $(2,1)$, we obtain the claim.
       \begin{figure}[htbp]
	\includegraphics[scale=0.5]{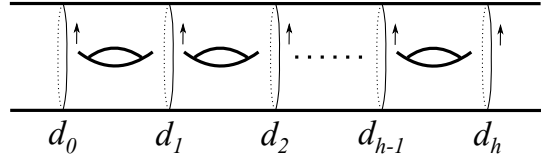}
	\caption{Loops $d_0,d_1,\dots,d_h$.}\label{bp}
       \end{figure}
\end{enumerate}
\end{proof}

\section{Preliminaries}\label{preliminary}

\subsection{On generators for $\g(N_g)$}\

McCarthy, Pinkall \cite{mp} and Gadgil, Pancholi \cite{gp} proved that
the natural homomorphism $\rho_2:\M(N_g)\to\Az$ is surjective.
Szepietowski \cite{s2} proved that $\g(N_g)$ is generated by
$Y$-homeomorphisms, and that $\g(N_g)$ is generated by involutions.
Therefore, $H_1(\g(N_g);\Z)$ is a $\Z/{2\Z}$-module.
The first author and Sato \cite{hs} showed that $H_1(\g(N_g);\Z)$ is the
$\Z/{2\Z}$-module of the rank $\binom{g}{3}+\binom{g}{2}$.

For $I=\{i_1,i_2,\dots,i_k\}\subset\{1,2,\dots,g\}$, we define an
oriented simple closed curve $\alpha_I$ as shown in Figure~\ref{loopI}.
For short, we denote $\alpha_{\{i\}}$ by $\alpha_i$.
We define $Y_{i_1;i_2,\dots,i_k}=Y_{\alpha_{i_1},\alpha_I}$,
$T_{i_1,\dots,i_k}=t_{\alpha_I}$ if $k$ is even.

\begin{figure}[htbp]
\includegraphics[scale=0.5]{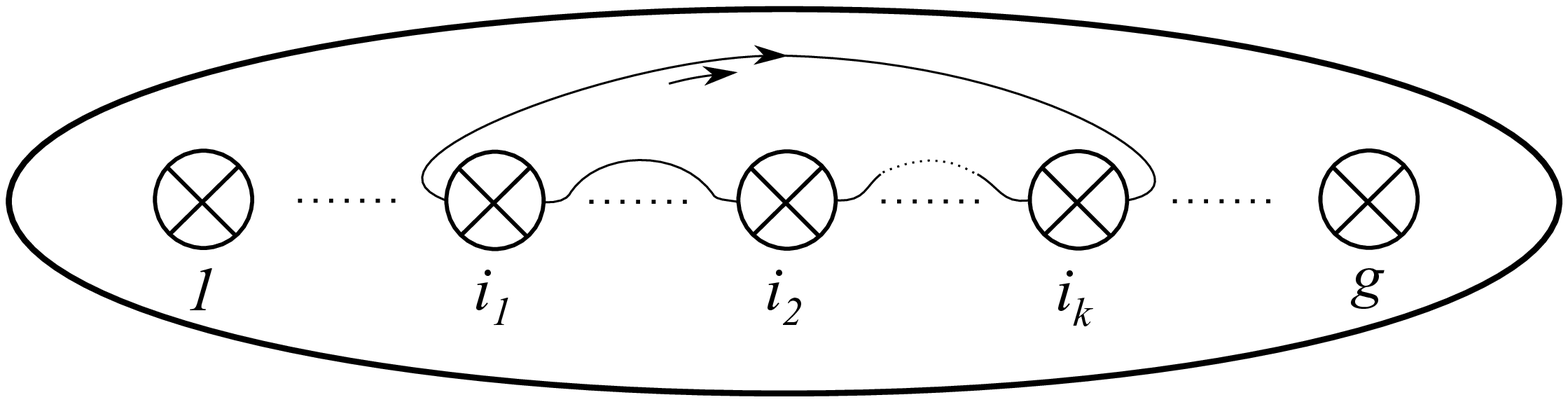}
\caption{The curve $\alpha_I$ for $I=\{i_1,i_2,\dots,i_k\}$.}\label{loopI}
\end{figure}

Szepietowski \cite{s3} gave a generating set for $\g(N_g)$ as follows.

\begin{thm}[\cite{s3}]
For $g\geq4$, $\g(N_g)$ is generated by the following elements.
\begin{enumerate}
 \item $Y_{i;j}$ for $1\leq{i}\leq{g-1}$, $1\leq{j}\leq{g}$ and
       $i\neq{j}$,
 \item $T_{i,j,k,l}^2$ for $1\leq{i<j<k<l}\leq{g}$.
\end{enumerate}
\end{thm}

In addition, the first author and Sato \cite{hs} gave a minimal
generating set for $\g(N_g)$ as follows.

\begin{thm}[\cite{hs}]\label{hs}
For $g\geq 4$, $\g(N_g)$ is generated by the following elements.
\begin{enumerate}
 \item $Y_{i;j}$ for $1\leq{i}\leq{g-1}$, $1\leq{j}\leq{g}$ and
       $i\neq{j}$,
 \item $T_{1,j,k,l}^2$ for $1<j<k<l\leq{g}$.
\end{enumerate}
\end{thm}

\subsection{On $\ker\Phi_g$ and $\g(g-1)$}\label{3.2}\

McCarthy and Pinkall claimed that $\ker\Phi_g$ is isomorphic to
$\g(g-1)$ in their unpublished preprint \cite{mp}.
In this subsection, we refer their result and proof.

Let $c_i=[\alpha_i]\in{H_1(N_g;\Z)}$ for $1\leq{i}\leq{g}$, and let
$c=c_1+c_2+\cdots+c_g(=[\alpha_{\{1,\dots,g\}}])$.
Then, by a natural handle decomposition as shown in Figure~\ref{handle},
as a $\Z$-module, $H_1(N_g;\Z)$ has a presentation
$$H_1(N_g;\Z)=\la{c_1,c_2,\dots,c_g\mid2c=0}\ra.$$
As a $\Z$-module we have
\begin{eqnarray*}
H_1(N_g;\Z)/{\la{c}\ra}
&=&\la{c_1,c_2,\dots,c_g\mid{c=0}}\ra\\
&=&\la{c_1,c_2,\dots,c_{g-1}}\ra\\
&\cong&\Z^{g-1},
\end{eqnarray*}
where we settle that the last isomorphism sends $c_i$ to the $i$-th
canonical normal vector $e_i$ for $1\leq{i}\leq{g-1}$.
For $x\in{H_1(N_g;\Z)}$, we denote by $\overline{x}$ the image of $x$ by
the projection $H_1(N_g;\Z)\to\Z^{g-1}$.
Explicitly, for $x=\sum_{j=1}^{g}x_jc_j\in{H_1(N_g;\Z)}$, we have
$\overline{x}=\sum_{j=1}^{g-1}(x_j-x_g)e_j\in\Z^{g-1}$.
We regard ${\rm Aut}(H_1(N_g;\Z)/{\la{c}\ra})$ as $GL(g-1;\Z)$.

\begin{figure}[htbp]
\includegraphics[scale=0.5]{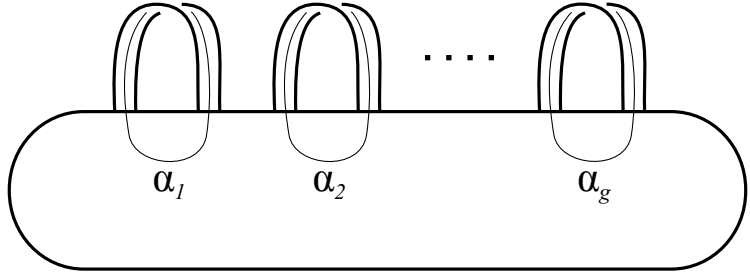}
\caption{A handle decomposition of $N_g$ whose cores of the $1$-handles
 are $\alpha_i$.}\label{handle}
\end{figure}

For $L\in{\rm Aut}(H_1(N_g;\Z))$, since $2L(c)=L(2c)=0$ and $c$ is the
only non-trivial element of $H_1(N_g;\Z)$ satisfying $2c=0$, we have
$L(c)=c$.
Hence $L\in{\rm Aut}(H_1(N_g;\Z))$ induces $\overline{L}\in{GL(g-1;\Z)}$.
More precisely, $\overline{L}$ is defined as
$\overline{L}(e_i)=\overline{L(c_i)}$.
By this correspondence, we obtain the following.

\begin{prop}
The correspondence $f:\ker\Phi_g\to\g(g-1)$ defined by
 $f(L)=\overline{L}$ is an isomorphism.
\end{prop}

\begin{proof}
We first show that $f(\ker\Phi_g)$ is in $\g(g-1)$ and $f$ is a
 homomorphism.
By the definition of $\Phi_g$, we have that $L(c_i)\equiv{c_i}\mod2$ for
$1\leq{i}\leq{g}$.
Hence $\overline{L}(e_i)=\overline{L(c_i)}\equiv\overline{c_i}=e_i\mod2$
 for $1\leq{i}\leq{g-1}$.
Therefore we have $f(L)\in\g(g-1)$.
In addition, for $L,L'\in\ker\Phi_g$, we see
\begin{eqnarray*}
\overline{LL'}(e_i)
&=&\overline{LL'(c_i)}\\
&=&\overline{L(L'(c_i))}\\
&=&\overline{L}(\overline{L'(c_i)})\\
&=&\overline{L}(\overline{L'}(e_i)).
\end{eqnarray*}
Thus, $f$ is a homomorphism.

We next show the injectivity of $f$.
For $L\in\ker\Phi_g$, suppose that $\overline{L}$ is the identity.
Then we have either $L(c_i)=c_i$ or $L(c_i)=c_i+c$.
By the definition of $\Phi_g$, we have that $L(c_i)\equiv{c_i}\mod2$ for 
$1\leq{i}\leq{g}$.
Hence $L$ is the identity.
Therefore $f$ is injective.

Finally we show the surjectivity of $f$.
For any $A=(a_{ij})\in\g(g-1)$, we define $\widetilde{A}\in\ker\Phi_g$
 to be 
\begin{eqnarray*}
\widetilde{A}(c_i)&=&
\left\{
\begin{array}{ll}
\sum_{j=1}^{g-1}a_{ji}c_j&(i\neq{g}),\\
\sum_{j=1}^{g-1}(1-\sum_{k=1}^{g-1}a_{jk})c_j+c_g&(i=g).
\end{array}
\right.
\end{eqnarray*}
Then we have that $\widetilde{A}(c)=c$ and,
since $a_{ii}$ is odd and $a_{ij}$ is even for $i\neq{j}$, 
$\widetilde{A}(c_i)\equiv{c_i}\mod2$.
Hence we have $\widetilde{A}\in\ker\Phi_g$.
In addition, we have $f(\widetilde{A})=A$.
Therefore $f$ is surjective.

Thus $f$ is an isomorphism.
\end{proof}

\subsection{On a presentation for $\g(g-1)$}\

For $1\leq{i,j}\leq{n}$ with $i\neq{j}$, let $E_{ij}$ denote the matrix
whose $(i,j)$ entry is $2$, diagonal entries are $1$ and others are $0$,
and let $F_i$ denote the matrix whose $(i,i)$ entry is $-1$, other
diagonal entries are $1$ and others are $0$.
It is known that $\g(n)$ is generated by $E_{ij}$ and $F_i$ (see
\cite{mp}).
In addition, a finite presentation for $\g(n)$ was independently given by
Fullarton \cite{f}, the second author \cite{k}, Margalit and Putman
recently.

\begin{thm}[cf. \cite{f}, \cite{k}]\label{kobayashi}
For $n\geq1$, $\g(n)$ has a finite presentation with generators
 $E_{ij}$ and $F_i$, for $1\leq{i,j}\leq{n}$, and with relators 
\begin{enumerate}
 \item $F_i^2 $,
 \item $(E_{ij}F_i)^2$,
       $(E_{ij}F_j)^2$,
       $(F_iF_j)^2$ (when $n\geq2$),
 \item \begin{enumerate}
	\item $[E_{ij},E_{ik}]$,
	      $[E_{ij},E_{kj}]$,
	      $[E_{ij},F_k ]$,
	      $[E_{ij},E_{ki}]E_{kj}^2$ 
	      (when $n\geq3$),
	\item $(E_{ji}E_{ij}^{-1}E_{kj}^{-1}E_{jk}E_{ik}E_{ki}^{-1})^2$
	      for $i<j<k$ (when $n\geq3$),
       \end{enumerate}
 \item $[E_{ij},E_{kl}]$ (when $n\geq4$),
\end{enumerate}
where $[X,Y]=X^{-1}Y^{-1}XY$ and $1\leq{i,j,k,l}\leq{n}$ are mutually different.
\end{thm}

For $1\leq{i}\leq{g-1}$ and $1\leq{j}\leq{g}$ with $i\neq{j}$, let 
$Y_{ij}=f((Y_{i;j})_\ast)$, where $\varphi_\ast\in\AZ$ means the
automorphism induced by $\varphi\in\M(N_g)$.
Then we have that $Y_{ij}=E_{ij}F_i$ if $j<g$ and $Y_{ig}=F_i$.
We now prove the following.

\begin{prop}\label{kobayashi2}
For $g-1\geq1$, $\g(g-1)$ has a finite presentation with generators
 $Y_{ij}$ for $1\leq{i}\leq{g-1}$ and $1\leq{j}\leq{g}$ with $i\neq{j}$,
 and with relators
\begin{enumerate}
 \item $Y_{ij}^2$ for $1\leq{i}\leq{g-1}$ and $1\leq{j}\leq{g}$,
 \item $[Y_{ik},Y_{jk}]$ for $1\leq{i,j}\leq{g-1}$ and $1\leq{k}\leq{g}$,
 \item $[Y_{ij},Y_{ik}Y_{jk}]$ for $1\leq{i,j}\leq{g-1}$ and $1\leq{k}\leq{g}$,
 \item $[Y_{ij},Y_{kl}]$ for $1\leq{i,k}\leq{g-1}$ and $1\leq{j,l}\leq{g}$,
 \item $(Y_{ij}Y_{ik}Y_{il})^2$ for $1\leq{i}\leq{g-1}$ and $1\leq{j,k,l}\leq{g}$,
 \item $(Y_{ji}Y_{ij}Y_{kj}Y_{jk}Y_{ik}Y_{ki})^2$ for $1\leq{i,j,k}\leq{g-1}$,
\end{enumerate}
where $[X,Y]=X^{-1}Y^{-1}XY$ and $i,j,k,l$ are mutually different.
\end{prop}

\begin{proof}
At first, we show that the relators of $\g(g-1)$ for
 Proposition~\ref{kobayashi2} are obtained from that for
 Theorem~\ref{kobayashi}.
By the relation $F_i^2=1$, we may identify $F_i^{-1}$ with $F_i$ in
 $\g(g-1)$.
\begin{enumerate}
 \item We have 
       $Y_{ij}^2=\left\{
        \begin{array}{ll}
	 (E_{ij}F_i)^2&(j<g),\\
	 F_i^2&(j=g).
	\end{array}
       \right.$
       Hence we obtain the relator $Y_{ij}^2$ in $\g(g-1)$.
 \item For $k<g$, we see
       \begin{eqnarray*}
	Y_{ik}Y_{jk}
	&=&\underset{(3)(a)}{\underline{E_{ik}F_iE_{jk}}}F_j\\
	&=&E_{jk}\underset{(2),~(3)(a)}{\underline{E_{ik}F_iF_j}}\\
	&=&E_{jk}F_jE_{ik}F_i\\
	&=&Y_{jk}Y_{ik}.
       \end{eqnarray*}
       In addition, we see
       \begin{eqnarray*}
	Y_{ig}Y_{jg}
	&=&F_iF_j\\
	&\underset{(2)}{=}&F_jF_i\\
	&=&Y_{jg}Y_{ig}.
       \end{eqnarray*}
       Hence we obtain the relator $[Y_{ik},Y_{jk}]$ in $\g(g-1)$.
 \item For $k<g$, we see
       \begin{eqnarray*}
	Y_{ij}Y_{ik}Y_{jk}
	&=&E_{ij}\underset{(2)}{\underline{F_iE_{ik}}}\underset{(3)(a),~(2)}{\underline{F_iE_{jk}F_j}}\\
	&=&\underset{(3)(a),~(2)}{\underline{E_{ij}E_{ik}^{-1}F_i}}E_{jk}F_jF_i\\
	&=&E_{ik}^{-1}F_i\underset{(3)(a)}{\underline{E_{ij}^{-1}E_{jk}}}F_jF_i\\
	&=&E_{ik}^{-1}\underset{(2)}{\underline{F_iE_{ik}^{-2}}}E_{jk}\underset{(2)}{\underline{E_{ij}^{-1}F_j}}F_i\\
	&=&E_{ik}F_iE_{jk}F_jE_{ij}F_i\\
	&=&Y_{ik}Y_{jk}Y_{ij}.
       \end{eqnarray*}
       In addition, we see
       \begin{eqnarray*}
	Y_{ij}Y_{ig}Y_{jg}
	&=&\underset{(2)}{\underline{E_{ij}F_i}}~\underset{(2)}{\underline{F_iF_j}}\\
	&=&F_i\underset{(2)}{\underline{E_{ij}^{-1}F_j}}F_i\\
	&=&F_iF_jE_{ij}F_i\\
	&=&Y_{ig}Y_{jg}Y_{ij}.
       \end{eqnarray*}
       Hence we obtain the relator $[Y_{ij},Y_{ik}Y_{jk}]$ in
       $\g(g-1)$.
 \item For $j,l<g$, we see
       \begin{eqnarray*}
	Y_{ij}Y_{kl}
	&=&E_{ij}\underset{(3)(a),~(2)}{\underline{F_iE_{kl}F_k}}\\
	&=&\underset{(4),~(3)(a)}{\underline{E_{ij}E_{kl}F_k}}F_i\\
	&=&E_{kl}F_kE_{ij}F_i\\
	&=&Y_{kl}Y_{ij}.
       \end{eqnarray*}
       In addition, we see
       \begin{eqnarray*}
	Y_{ij}Y_{kg}
	&=&E_{ij}F_iF_k\\
	&\underset{(2),~(3)(a)}{=}&F_kE_{ij}F_i\\
	&=&Y_{kg}Y_{ij}.
       \end{eqnarray*}
       Hence we obtain the relator $[Y_{ij},Y_{kl}]$ in $\g(g-1)$.
 \item For the relator $(Y_{ij}Y_{ik}Y_{il})^2$, by the fact
       $Y_{im}=Y_{im}^{-1}$, applying conjugations and taking their
       inverses, it suffices to consider the case $j<k<l$.
       For $l<g$, we see
       \begin{eqnarray*}
	Y_{ij}Y_{ik}Y_{il}
	&=&E_{ij}\underset{(2)}{\underline{F_iE_{ik}}}~\underset{(2)}{\underline{F_iE_{il}}}F_i\\
	&=&\underset{(3)(a),~(2)}{\underline{E_{ij}E_{ik}^{-1}F_i}}E_{il}^{-1}F_iF_i\\
	&=&E_{ik}^{-1}F_i\underset{(3)(a),~(2)}{\underline{E_{ij}^{-1}E_{il}^{-1}F_i}}F_i\\
	&=&E_{ik}^{-1}\underset{(2)}{\underline{F_iE_{il}^{-1}}}F_iE_{ij}F_i\\
	&=&\underset{(3)(a),~(2)}{\underline{E_{ik}^{-1}E_{il}F_i}}F_iE_{ij}F_i\\
	&=&E_{il}F_iE_{ik}F_iE_{ij}F_i\\
	&=&Y_{il}Y_{ik}Y_{ij}\\
	&=&Y_{il}^{-1}Y_{ik}^{-1}Y_{ij}^{-1}.
       \end{eqnarray*}
       In addition, we see
       \begin{eqnarray*}
	Y_{ij}Y_{ik}Y_{ig}
	&=&\underset{(2),~(3)(a)}{\underline{E_{ij}F_iE_{ik}F_i}}F_i\\
	&=&F_iE_{ik}F_iE_{ij}F_i\\
	&=&Y_{ig}Y_{ik}Y_{ij}\\
	&=&Y_{ig}^{-1}Y_{ik}^{-1}Y_{ij}^{-1}.
       \end{eqnarray*}
       Hence we obtain the relator  $(Y_{ij}Y_{ik}Y_{il})^2$ in
       $\g(g-1)$.
 \item We see 
       \begin{eqnarray*}
	Y_{ji}Y_{ij}\cdot Y_{kj}Y_{jk}\cdot Y_{ik}Y_{ki}
	&=&(E_{ji}F_jE_{ij}F_i)(E_{kj}F_kE_{jk}F_j)(E_{ik}\underline{F_i}E_{ki}F_k)\\
	&\underset{(2)}{=}&(E_{ji}F_jE_{ij}F_i)(E_{kj}F_kE_{jk}\underline{F_j})(E_{ik}E_{ki}^{-1})F_iF_k\\
	&\underset{(3)(a)}{=}&(E_{ji}F_jE_{ij}F_i)(E_{kj}\underline{F_k}E_{jk})(E_{ik}E_{ki}^{-1})F_jF_iF_k\\
	&\underset{(2)}{=}&(E_{ji}F_jE_{ij}\underline{F_i})(E_{kj}E_{jk}^{-1})(E_{ik}^{-1}E_{ki})F_kF_jF_iF_k\\
	&\underset{(3)(a),~(2)}{=}&(E_{ji}\underline{F_j}E_{ij})(E_{kj}E_{jk}^{-1})(E_{ik}E_{ki}^{-1})F_iF_kF_jF_iF_k\\
	&\underset{(2),~(3)(a)}{=}&(E_{ji}E_{ij}^{-1})(E_{kj}^{-1}E_{jk})(E_{ik}E_{ki}^{-1})F_jF_iF_kF_jF_iF_k\\
	&\underset{(1),~(2)}{=}&(E_{ji}E_{ij}^{-1})(E_{kj}^{-1}E_{jk})(E_{ik}E_{ki}^{-1}).
       \end{eqnarray*}
       By (3)~(b) of Theorem~\ref{kobayashi}, we obtain the relator  
       $(Y_{ji}Y_{ij}\cdot Y_{kj}Y_{jk}\cdot Y_{ik}Y_{ki})^2$ in
       $\g(g-1)$.
\end{enumerate}
Next, we show that the relators of $\g(g-1)$ for
 Theorem~\ref{kobayashi} are obtained from that for
 Proposition~\ref{kobayashi2}.
By the relation $Y_{ij}^2=1$, we may identify $Y_{ij}^{-1}$ with
 $Y_{ij}$ in $\g(g-1)$.
Note that $E_{ij}=Y_{ij}Y_{ig}$ and $F_i=Y_{ig}$.
\begin{enumerate}
 \item Since $F_i^2=Y_{ig}^2$,
       we have the relator $F_i^2$ in $\g(g-1)$.
 \item Since $E_{ij}F_i=Y_{ij}$, we have the relator $(E_{ij}E_i)^2$ in
       $\g(g-1)$.
       We see 
       \begin{eqnarray*}
	(E_{ij}F_j)^2
	&=&\underset{(1)}{\underline{Y_{ij}}}~\underset{(1),~(2)}{\underline{Y_{ig}Y_{jg}}}\cdot{}Y_{ij}Y_{ig}Y_{jg}\\
	&=&Y_{ij}^{-1}(Y_{ig}Y_{jg})^{-1}\cdot Y_{ij}Y_{ig}Y_{jg}\\
	&=&[Y_{ij},Y_{ig}Y_{jg}].
       \end{eqnarray*}
       Hence we obtain the relator $(E_{ij}F_j)^2$ in $\g(g-1)$.
       We see 
       \begin{eqnarray*}
	(F_iF_j)^2
	&=&Y_{ig}\underset{(2)}{\underline{Y_{jg}Y_{ig}}}Y_{jg}\\
	&=&Y_{ig}^2Y_{jg}^2.
       \end{eqnarray*}
       Hence we obtain the relator $(F_iF_j)^2$ in $\g(g-1)$.
 \item \begin{enumerate}
	\item We see 
	      \begin{eqnarray*}
	       E_{ij}E_{ik}
		&=&\underset{(5)}{\underline{Y_{ij}Y_{ig}Y_{ik}}}Y_{ig}\\
	       &=&Y_{ik}Y_{ig}Y_{ij}Y_{ig}\\
	       &=&E_{ik}E_{ij}.
	      \end{eqnarray*}
	      Hence we obtain the relator $[E_{ij},E_{ik}]$ in $\g(g-1)$.
	      We see 
	      \begin{eqnarray*}
	       E_{ij}E_{kj}
	       &=&Y_{ij}\underset{(4),~(2)}{\underline{Y_{ig}Y_{kj}Y_{kg}}}\\
	       &=&\underset{(2),~(4)}{\underline{Y_{ij}Y_{kj}Y_{kg}}}Y_{ig}\\
	       &=&Y_{kj}Y_{kg}Y_{ij}Y_{ig}\\
	       &=&E_{kj}E_{ij}.
	      \end{eqnarray*}
	      Hence we obtain the relator $[E_{ij},E_{kj}]$ in $\g(g-1)$.
	      We see 
	      \begin{eqnarray*}
	       E_{ij}F_k
	       &=&Y_{ij}Y_{ig}Y_{kg}\\
	       &\underset{(2),~(4)}{=}&Y_{kg}Y_{ij}Y_{ig}\\
	       &=&F_kE_{ij}.
	      \end{eqnarray*}
	      Hence we obtain the relator $[E_{ij},F_k]$ in $\g(g-1)$.
	      We see 
	      \begin{eqnarray*}
	       E_{ij}E_{ki}E_{kj}^2
		&=&Y_{ij}Y_{ig}\underset{(5)}{\underline{Y_{ki}Y_{kg}Y_{kj}}}Y_{kg}Y_{kj}Y_{kg}\\
	       &=&Y_{ij}\underset{(4),~(2)}{\underline{Y_{ig}Y_{kj}Y_{kg}}}Y_{ki}Y_{kg}Y_{kj}Y_{kg}\\
	       &=&\underset{(2)}{\underline{Y_{ij}Y_{kj}}}~\underset{(3)}{\underline{Y_{kg}Y_{ig}Y_{ki}}}Y_{kg}Y_{kj}Y_{kg}\\
	       &=&\underset{(3)}{\underline{Y_{kj}Y_{ij}Y_{ki}}}Y_{kg}\underset{(2),~(4)}{\underline{Y_{ig}Y_{kg}Y_{kj}Y_{kg}}}\\
	       &=&Y_{ki}Y_{kj}Y_{ij}\underset{(1)}{\underline{Y_{kg}Y_{kg}}}Y_{kj}Y_{kg}Y_{ig}\\
	       &=&Y_{ki}Y_{kj}\underset{(2),~(4)}{\underline{Y_{ij}Y_{kj}Y_{kg}}}Y_{ig}\\
	       &=&Y_{ki}\underset{(1)}{\underline{Y_{kj}Y_{kj}}}Y_{kg}Y_{ij}Y_{ig}\\
	       &=&Y_{ki}Y_{kg}Y_{ij}Y_{ig}\\
	       &=&E_{ki}E_{ij}.
	      \end{eqnarray*}
	      Hence we obtain the relator $[E_{ij},E_{ki}]E_{kj}^2$ in
	      $\g(g-1)$.
	\item Since we already obtained relators (1), (2) and (a) of (3)
	      for Theorem~\ref{kobayashi}, using these relators, we see
	      \begin{eqnarray*}
	       (E_{ji}E_{ij}^{-1})(E_{kj}^{-1}E_{jk})(E_{ik}E_{ki}^{-1})
	       &=&(E_{ji}E_{ij}^{-1})(E_{kj}^{-1}E_{jk})(E_{ik}E_{ki}^{-1})F_jF_iF_kF_jF_iF_k\\
	       &=&(E_{ji}F_jE_{ij})(E_{kj}E_{jk}^{-1})(E_{ik}E_{ki}^{-1})F_iF_kF_jF_iF_k\\
	       &=&(E_{ji}F_jE_{ij}F_i)(E_{kj}E_{jk}^{-1})(E_{ik}^{-1}E_{ki})F_kF_jF_iF_k\\
	       &=&(E_{ji}F_jE_{ij}F_i)(E_{kj}F_kE_{jk})(E_{ik}E_{ki}^{-1})F_jF_iF_k\\
	       &=&(E_{ji}F_jE_{ij}F_i)(E_{kj}F_kE_{jk}F_j)(E_{ik}E_{ki}^{-1})F_iF_k\\
	       &=&(E_{ji}F_jE_{ij}F_i)(E_{kj}F_kE_{jk}F_j)(E_{ik}F_iE_{ki}F_k)\\
	       &=&Y_{ji}Y_{ij}\cdot Y_{kj}Y_{jk}\cdot Y_{ik}Y_{ki}.
	      \end{eqnarray*}
	      By Proposition~\ref{kobayashi2}~(6), we obtain the relator
	      $(E_{ji}E_{ij}^{-1}E_{kj}^{-1}E_{jk}E_{ik}E_{ki}^{-1})^2$
	      in $\g(g-1)$.
      \end{enumerate}
 \item We see 
       \begin{eqnarray*}
	E_{ij}E_{kl}
	&=&Y_{ij}\underset{(2),~(4)}{\underline{Y_{ig}Y_{kl}Y_{kg}}}\\
	&=&\underset{(4)}{\underline{Y_{ij}Y_{kl}Y_{kg}}}Y_{ig}\\
	&=&Y_{kl}Y_{kg}Y_{ij}Y_{ig}\\
	&=&E_{kl}E_{ij}.
       \end{eqnarray*}
       Hence we obtain the relator $[E_{ij},E_{kl}]$ in $\g(g-1)$.
\end{enumerate}
Thus, we complete the proof.
\end{proof}

\section{A normal generating set for $\I(N_g)$}\label{norgen}

Let $f:\ker\Phi_g\to\g(g-1)$ be the isomorphism introduced in
Subsection~\ref{3.2}.
In order to obtain a presentation for $\g(g-1)$ whose generators are
$Y_{ij}:=f((Y_{i;j})_\ast)$, $T_{1jkl}^2:=f((T_{1,j,k,l}^2)_\ast)$, we
need to express $T_{1jkl}^2$ as a product of $Y_{ij}$'s.

For $I=\{i_1,i_2,\dots,i_k\}\subset\{1,2,\dots,g\}$, we define a simple
closed curve $\alpha_I'$ as shown in Figure~\ref{loopI2} and
$T_{i,j,k,l}'=t_{\alpha_{\{i,j,k,l\}}'}$.
Note that $T_{i,j,k,l}T_{i,j,k,l}'^{-1}$ is a BP map.
In addition, for $1\leq{m}\leq{g}$ with $m\neq{i,j,k,l}$, there exist
$A$-circles $\beta_{m,i}$, $\beta_{m,j}$, $\beta_{m,k}$ and
$\beta_{m,l}$ intersecting $\alpha_m$ at only one point such that
$$T_{i,j,k,l}^{-1}T_{i,j,k,l}'^{-1}=\prod_{m\neq{i,j,k,l}}Y_{\alpha_m,\beta_{m,l}}Y_{\alpha_m,\beta_{m,k}}Y_{\alpha_m,\beta_{m,j}}Y_{\alpha_m,\beta_{m,i}}.$$
For example, when $(i,j,k,l)$ is $(1,2,3,4)$, for $m\geq5$ and
$t=1,2,3,4$ we have
\begin{eqnarray*}
Y_{\alpha_m,\beta_{m,t}}&=&
\left\{
\begin{array}{ll}
Y_{m;t}^{-1}&(t=1,2,3),\\
Y_{m;m-1}^{-2}\cdots{}Y_{m;6}^{-2}Y_{m;5}^{-2}Y_{m;4}^{-1}&(t=4).
\end{array}
\right.
\end{eqnarray*}
Therefore, we have
\begin{eqnarray*}
T_{1jkl}^{-2}
&=&f((T_{1,j,k,l}^{-1}T_{1,j,k,l}'^{-1})_\ast)\\
&=&f((\prod_{m\neq{1,j,k,l}}Y_{\alpha_m,\beta_{m,l}}Y_{\alpha_m,\beta_{m,k}}Y_{\alpha_m,\beta_{m,j}}Y_{\alpha_m,\beta_{m,1}})_\ast).
\end{eqnarray*}
Note that any $Y$-homeomorphism $Y_{\alpha_m,\beta}$ is a product of
some $Y_{i;j}$ for $1\leq{i}\leq{g-1}$ and $1\leq{j}\leq{g}$ with
$i\neq{j}$.
For example, we have
\begin{align}\label{append(II)}
Y_{g;i}=&
(Y_{1;2}^2\cdots{}Y_{1;g-1}^2Y_{1;i}^{-1}Y_{1;g})\cdots(Y_{i-1;i}^2\cdots{}Y_{i-1;g-1}^2Y_{i-1;i}^{-1}Y_{i-1;g})\tag{I\!I}\\
&\cdot(Y_{i+1;i+2}^2\cdots{}Y_{i+1;g-1}^2Y_{i+1;i}^{-1}Y_{i+1;g}Y_{i+1;i}^2)\cdots(Y_{g-2;g-1}^2Y_{g-2;i}^{-1}Y_{g-2;g}Y_{g-2;i}^2)\notag\\
&\cdot(Y_{g-1;i}^{-1}Y_{g-1;g}Y_{g-1;i}^2)Y_{i;g}.\notag
\end{align}
The equation (\ref{append(II)}) will be shown in
Appendix~\ref{appendix}.
Thus $T_{1jkl}^2$ can be expressed as a product of $Y_{ij}$'s.

\begin{figure}[htbp]
\includegraphics[scale=0.5]{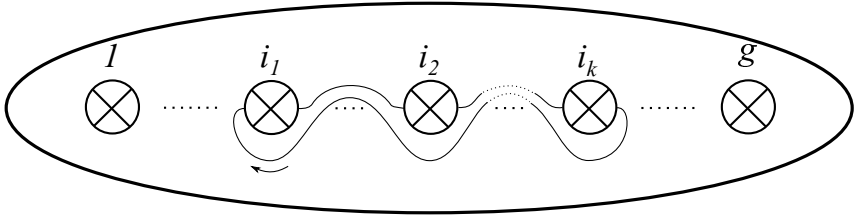}
\caption{The curve $\alpha_I'$ for $I=\{i_1,i_2,\dots,i_k\}$.}\label{}\label{loopI2}
\end{figure}

\subsection{A normal generating set for $\I(N_g)$ in $\g(N_g)$}\

Let $G$ be a group.
For $x_1,x_2,\dots,x_n\in{G}$, we say that $N$ is {\it normally
generated} by $x_1,x_2,\dots,x_n$ in $G$ if $N$ is a minimal normal
subgroup of $G$ which contains $x_1,x_2,\dots,x_n$.

In this subsection, we prove the following.

\begin{prop}\label{normalgen}
For $g\geq 4$, $\I(N_g)$ is normally generated by the following elements
 in $\g(N_g)$.
\begin{enumerate}
 \item $Y_{i;j}^2$ for $1\leq{i}\leq{g-1}$ and $1\leq{j}\leq{g}$,
 \item $[Y_{i;k},Y_{j;k}]$ for $1\leq{i,j}\leq{g-1}$ and $1\leq{k}\leq{g}$,
 \item $[Y_{i;j},Y_{i;k}Y_{j;k}]$ for $1\leq{i}\leq{g-1}$ and $1\leq{j,k}\leq{g}$,
 \item $[Y_{i;j},Y_{k;l}]$ for $1\leq{i,k}\leq{g-1}$ and $1\leq{j,l}\leq{g}$,
 \item $(Y_{i;j}Y_{i;k}Y_{i;l})^2$ for $1\leq{i}\leq{g-1}$ and $1\leq{j,k,l}\leq{g}$,
 \item $(Y_{j;i}Y_{i;j}Y_{k;j}Y_{j;k}Y_{i;k}Y_{k;i})^2$ for $1\leq{i,j,k}\leq{g-1}$,
 \item $T_{1,j,k,l}^2(\prod_{m\neq{1,j,k,l}}Y_{\alpha_m,\beta_{m,l}}Y_{\alpha_m,\beta_{m,k}}Y_{\alpha_m,\beta_{m,j}}Y_{\alpha_m,\beta_{m,1}})$
       for $1<j<k<l\leq{g}$,
\end{enumerate}
where $i,j,k,l$ are mutually different.
\end{prop}

Let $\Gamma=\la{g_1,g_2,\dots,g_n\mid{r_1,r_2,\dots,r_k}}\ra$ be a
finitely presented group, and let $G$ be a group generated by 
$\tilde{g}_1,\tilde{g}_2,\dots,\tilde{g}_n$.
For
$r_i=g_{i(1)}^{\varepsilon_1}g_{i(2)}^{\varepsilon_2}\cdots{g_{i(m)}^{\varepsilon_m}}$,
define
$\tilde{r}_i=\tilde{g}_{i(1)}^{\varepsilon_1}\tilde{g}_{i(2)}^{\varepsilon_2}\cdots\tilde{g}_{i(m)}^{\varepsilon_m}$,
where $\varepsilon_j=\pm1$.
Let $\widetilde{N}$ be a normal subgroup of $G$ normally generated by
$\tilde{r}_1,\tilde{r}_2,\dots,\tilde{r}_k$. 
We first prove the following.

\begin{lem}\label{lem1}
Let $\mu:G\to\Gamma$ be a homomorphism sending $\tilde{g}_i$ to $g_i$.
Then we have $\widetilde{N}=\ker\mu$.
\end{lem}

\begin{proof}
Let $F=\la{g_1,g_2,\dots,g_n}\ra$, and let $N$ be a normal subgroup of
 $F$ normally generated by $r_1,r_2,\dots,r_k$.
Let $\pi:F\to\Gamma$ be a natural epimorphism, and let $\nu:F\to{G}$ be
 a homomorphism sending $g_i$ to $\tilde{g}_i$.
Since $\mu$ is surjective, we have $\pi=\mu\nu$.
Then we have the following short exact sequences and commutative diagram.

$$\xymatrix{
1\ar[r]&N\ar[r]\ar[d]^{\nu|_N}&F\ar[r]^{\pi}\ar[d]^{\nu}&\Gamma\ar[r]\ar@{=}[d]&1\\
1\ar[r]&\ker\mu\ar[r]&G\ar[r]^{\mu}&\Gamma\ar[r]&1
}$$

For any $\widetilde{R}\in\ker\mu$, there exists $R\in{F}$ such that
 $\nu(R)=\widetilde{R}$.
Then we have that $\pi(R)=\mu\nu(R)=\mu(\widetilde{R})=1$.
Hence we have $R\in\ker\pi=N$.
Therefore $\nu|_N:N\to\ker\mu$ is surjective.
Since $\nu(N)=\widetilde{N}$, we conclude that $\widetilde{N}=\ker\mu$. 
\end{proof}

\begin{proof}[Proof of Proposition~\ref{normalgen}]
Let $F$ be a free group of the rank $\binom{g}{3}+\binom{g}{2}$
generated by $Y_{ij}$ for $1\leq i\leq g-1$ and $1\leq j\leq g$ with
 $i\neq j$, and $T_{1jkl}^2$ for $1<j<k<l\leq g$, where
 $T_{1jkl}=f((T_{1,j,k,l})_*)$, and let $N$ be a normal subgroup of $F$
 normally generated by followings.
\begin{enumerate}
 \item $Y_{ij}^2$ for $1\leq i\leq g-1$ and $1\leq j\leq g$,
 \item $[Y_{ik},Y_{jk}]$ for $1\leq i,j\leq g-1$ and $1\leq k\leq g$,
 \item $[Y_{ij},Y_{ik}Y_{jk}]$ for $1\leq i\leq g-1$ and $1\leq j,k\leq g$,
 \item $[Y_{ij},Y_{kl}]$ for $1\leq i,k\leq g-1$ and $1\leq j,l\leq g$,
 \item $(Y_{ij}Y_{ik}Y_{il})^2$ for $1\leq i\leq g-1$ and $1\leq j,k,l\leq g$,
 \item $(Y_{ji}Y_{ij}Y_{kj}Y_{jk}Y_{ik}Y_{ki})^2$ for $1\leq i,j,k\leq g-1$,
 \item $T_{1jkl}^2f((\prod_{m\neq{1,j,k,l}}Y_{\alpha_m,\beta_{m,l}}Y_{\alpha_m,\beta_{m,k}}Y_{\alpha_m,\beta_{m,j}}Y_{\alpha_m,\beta_{m,1}})_*)$
       for $1<j<k<l\leq g$,
\end{enumerate}
where $i,j,k,l$ are mutually different.
By Proposition~\ref{kobayashi2} and the fact that
 $$f((\prod_{m\neq{1,j,k,l}}Y_{\alpha_m,\beta_{m,l}}Y_{\alpha_m,\beta_{m,k}}Y_{\alpha_m,\beta_{m,j}}Y_{\alpha_m,\beta_{m,1}})_*)=T_{1jkl}^{-2},$$
 $\g(g-1)$ is the quotient of $F$ by $N$.
Let $\widetilde{N}$ be the normal subgroup of $\g(N_g)$ normally
 generated by followings.
\begin{enumerate}
 \item $Y_{i;j}^2$ for $1\leq i\leq g-1$ and $1\leq j\leq g$,
 \item $[Y_{i;k},Y_{j;k}]$ for $1\leq i,j\leq g-1$ and $1\leq k\leq g$,
 \item $[Y_{i;j},Y_{i;k}Y_{j;k}]$ for $1\leq i\leq g-1$ and $1\leq j,k\leq g$,
 \item $[Y_{i;j},Y_{k;l}]$ for $1\leq i,k\leq g-1$ and $1\leq j,l\leq g$,
 \item $(Y_{i;j}Y_{i;k}Y_{i;l})^2$ for $1\leq i\leq g-1$ and $1\leq j,k,l\leq g$,
 \item $(Y_{j;i}Y_{i;j}Y_{k;j}Y_{j;k}Y_{i;k}Y_{k;i})^2$ for $1\leq i,j,k\leq g-1$,
 \item $T_{1,j,k,l}^2(\prod_{m\neq{1,j,k,l}}Y_{\alpha_m,\beta_{m,l}}Y_{\alpha_m,\beta_{m,k}}Y_{\alpha_m,\beta_{m,j}}Y_{\alpha_m,\beta_{m,1}})$
       for $1<j<k<l\leq g$,
\end{enumerate}
where $i,j,k,l$ are mutually different.
Let $\pi:F\to\g(g-1)$ be a natural epimorphism, and let
 $\nu:F\to\g(N_g)$ be a homomorphism sending $Y_{ij}$ to $Y_{i;j}$,
 $T_{1jkl}^2$ to $T_{1,j,k,l}^2$.
Then we have the following short exact sequences and commutative
 diagram.
$$\xymatrix{
1\ar[r]&N\ar[r]\ar[d]^{\nu|_{N}}&F\ar[r]^{\pi}\ar[d]^{\nu}&\g(g-1)\ar[r]\ar@{->}[d]^{f^{-1}}&1\\
1\ar[r]&\ker\rho'\ar[r]&\g(N_g)\ar[r]^{\rho'}&\ker\Phi_g\ar[r]&1
}$$
By Lemma~\ref{lem1}, we have $\widetilde{N}=\ker\rho'$.
On the other hand, since $\I(N_g)=\ker\rho'$, we obtain the claim.
\end{proof}

\subsection{A normal generating set for $\I(N_g)$ in $\M(N_g)$}\

In this subsection, we prove the following.

\begin{prop}\label{normalgen2}
For $g\geq 4$, $\I(N_g)$ is normally generated by the following elements
 in $\M(N_g)$.
\begin{enumerate}
 \item $Y_{1;2}^2$,
 \item $[Y_{1;3},Y_{2;3}]$,
 \item $[Y_{1;2},Y_{1;3}Y_{2;3}]$, $[Y_{1;3},Y_{1;2}Y_{3;2}]$,
 \item $[Y_{1;2},Y_{3;4}]$, $[Y_{1;3},Y_{2;4}]$,
 \item $(Y_{1;2}Y_{1;3}Y_{1;4})^2$,
 \item $(Y_{2;1}Y_{1;2}Y_{3;2}Y_{2;3}Y_{1;3}Y_{3;1})^2$,
 \item $T_{1,2,3,4}^2(\prod_{5\leq{m}\leq{g}}Y_{m;m-1}^{-2}\cdots{Y_{m;6}^{-2}}Y_{m;5}^{-2}Y_{m;4}^{-1}Y_{m;3}^{-1}Y_{m;2}^{-1}Y_{m;1}^{-1})$.
\end{enumerate}
\end{prop}

For $1\leq{i<j}\leq{g}$, let $c_{ij}$ be simple closed curve on $N_g$ as
shown in Figure~\ref{U}, and let $U_{i,j}$ be a diffeomorphism over
$N_g$ as shown in Figure~\ref{U}.
Note that $Y_{i;j}=U_{i,j}T_{i,j}$ and $U_{i,j}^2=t_{c_{ij}}=Y_{i;j}^2$.

\begin{figure}[htbp]
\includegraphics[scale=0.5]{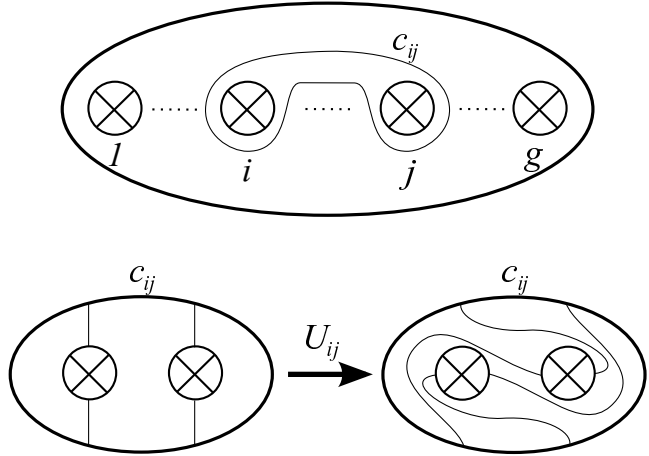}
\caption{The loop $c_{ij}$ and the diffeomorphism $U_{ij}$ over $N_g$.}\label{U}
\end{figure}

Since $\I(N_g)$ is a normal subgroup of $\g(N_g)$, the normal generating
set for $\I(N_g)$ in $\g(N_g)$ is a normal generating set for $\I(N_g)$
in $\M(N_g)$.
In addition, each normal generator for $\I(N_g)$ in
Proposition~\ref{normalgen} is conjugation of some normal generator for
$\I(N_g)$ in Proposition~\ref{normalgen2} by a product of some
$U_{i,j}$.
For example, we have
$$Y_{i;j}^2=(U_{i-i,i}\cdots{}U_{1,2})(U_{j-i,j}\cdots{}U_{2,3})Y_{1;2}^2(U_{2,3}^{-1}\cdots{}U_{j-1,j}^{-1})(U_{1,2}^{-1}\cdots{}U_{i-1,i}^{-1})$$
for $i<j$.
Thus we finish the proof of Proposition~\ref{normalgen2}.

\section{Proof of Theorem~\ref{thm}}\label{proofthm}

We put a point $*\in{N_{g-1}}$.
Let $\gamma_1,\gamma_2,\dots,\gamma_{g-1}$ be oriented loops on
$N_{g-1}$ starting at $*$ as shown in Figure~\ref{spin}.
Note that $\pi_1(N_{g-1},*)$ is generated by
$[\gamma_1],[\gamma_2],\dots,[\gamma_{g-1}]$.
For $1\leq{i}\leq{g}$, let $s_i:\pi_1(N_{g-1},*)\to\M(N_g)$ be the
crosscap pushing map defined in \cite{s3}, such that
$s_i([\gamma_j])=Y_{i;j}$ if $j<i$, $s_i([\gamma_j])=Y_{i;j+1}$ if
$j\geq{i}$.
We note that the crosscap pushing map is an anti-homomorphism.
Namely, we have
$s_i([\alpha][\beta])=s_i([\beta])s_i([\alpha])$ for
$[\alpha],[\beta]\in\pi_1(N_{g-1},*)$.
For $[\alpha]\in\pi_1(N_{g-1},*)$, $s_i([\alpha])$ is a
$Y$-homeomorphism if $\alpha$ is an $M$-circle, $s_i([\alpha])$ is a
product of two Dehn twists if $\alpha$ is an $A$-circle.

\begin{figure}[htbp]
\includegraphics[scale=0.5]{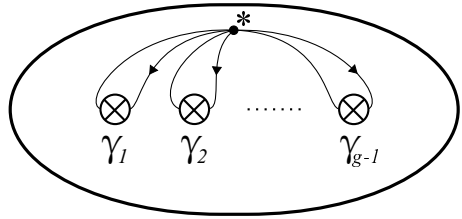}
\caption{The loops $\gamma_1,\gamma_2,\dots,\gamma_{g-1}$.}\label{spin}
\end{figure}

We have the following as a corollary of Proposition~\ref{normalgen2}.

\begin{cor}\label{normalgen3}
For $g\geq 4$, $\I(N_g)$ is normally generated by the following elements
 in $\M(N_g)$.
\begin{enumerate}
 \item $Y_{1;2}^2$,
 \item $[Y_{2;3},Y_{1;3}^{-1}]$,
 \item $[Y_{1;2},Y_{1;3}Y_{2;3}]$,
       $[Y_{1;3},Y_{3;2}Y_{1;2}]$,
 \item $[Y_{1;2},Y_{3;4}]$,
       $[Y_{2;3}^{-2}Y_{1;3}Y_{2;3}^2,Y_{2;4}]$,
 \item $(Y_{1;2}Y_{1;3}Y_{1;4})^2$,
 \item $(Y_{2;1}^{-1}Y_{1;2}Y_{3;2}Y_{2;3}^{-1}Y_{1;3}^{-1}Y_{3;1})^2$,
 \item $T_{1,2,3,4}^2(\prod_{5\leq{m}\leq{g}}Y_{m;m-1}^{-2}\cdots{Y_{m;6}^{-2}}Y_{m;5}^{-2}Y_{m;4}^{-1}Y_{m;3}^{-1}Y_{m;2}^{-1}Y_{m;1}^{-1})$.
\end{enumerate}
\end{cor}

We now prove Theorem~\ref{thm}.

\begin{proof}[Proof of Theorem~\ref{thm}]
We show that each normal generator for $\I(N_g)$ in
Corollary~\ref{normalgen3} is a product of BSCC maps and BP maps.
\begin{enumerate}
 \item We have that $Y_{1;2}^2=s_1([\gamma_1^2])=t_{c_1}t_{c_2}$, where
       $c_1$ and $c_2$ are simple closed curves as shown in
       Figure~\ref{(1)}.
       Note that $t_{c_1}$ is a BSCC map of type $(1,2)$.
       Since $c_2$ bounds a ${\rm M\ddot{o}bius}$ band, $t_{c_2}$ is
       trivial.
       Hence we have that $Y_{1;2}^2$ is a BSCC map of type $(1,2)$.
       \begin{figure}[htbp]
	\includegraphics[scale=0.5]{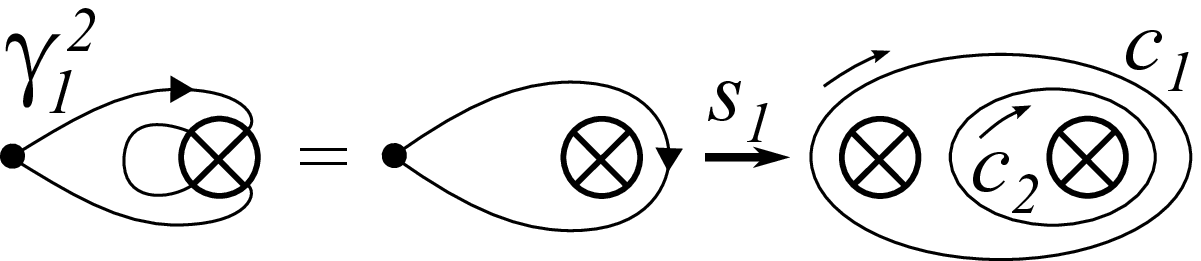}
	\caption{}\label{(1)}
       \end{figure}
 \item We have that
       $Y_{1;3}Y_{2;3}Y_{1;3}^{-1}=Y_{1;3}s_2([\gamma_2])Y_{1;3}^{-1}=s_2([\gamma_1^2\gamma_2])$
       (see Figure~\ref{(2)}).
       Hence we have that
       \begin{eqnarray*}
	[Y_{2;3},Y_{1;3}^{-1}]
	&=&
	Y_{2;3}^{-1}Y_{1;3}Y_{2;3}Y_{1;3}^{-1}\\
	&=&
	s_2([\gamma_2^{-1}])s_2([\gamma_1^2\gamma_2])\\
	&=&
	s_2([\gamma_1^2\gamma_2][\gamma_2^{-1}])\\
	&=&
	s_2([\gamma_1^2])\\
	&=&
	t_{c_1}t_{c_2},
       \end{eqnarray*}
       where $c_1$ and $c_2$ are simple closed curves as shown in
       Figure~\ref{(2)}.
       Similar to (1), we have that $[Y_{2;3},Y_{1;3}^{-1}]$ is a BSCC
       map of type $(1,2)$.
       \begin{figure}[htbp]
	\includegraphics[scale=0.5]{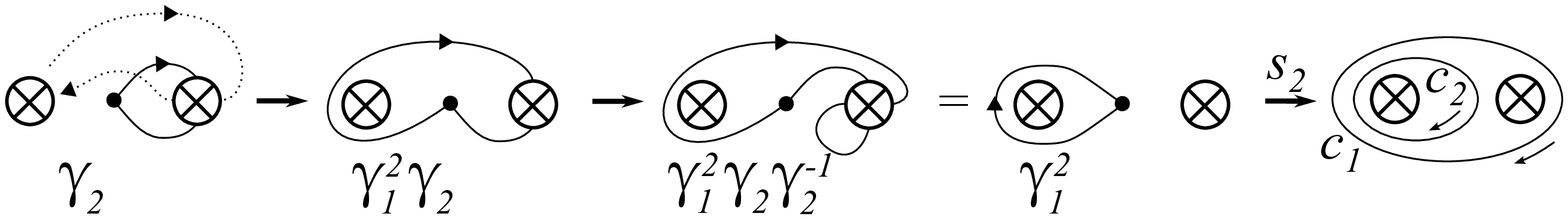}
	\caption{}\label{(2)}
       \end{figure}
 \item We have that
       \begin{eqnarray*}
	Y_{2;3}^{-1}Y_{1;3}^{-1}Y_{1;2}Y_{1;3}Y_{2;3}
	&=&
	Y_{2;3}^{-1}Y_{1;3}^{-1}s_1([\gamma_1])Y_{1;3}Y_{2;3}\\
	&=&
	Y_{2;3}^{-1}s_1([\gamma_2\gamma_1\gamma_2^{-1}])Y_{2;3}\\
	&=&
	s_1([\gamma_1^{-1}])\\
	&=&
	Y_{1;2}^{-1}
       \end{eqnarray*}
       (see Figure~\ref{(3)}~(a)).
       Hence we have that
       $[Y_{1;2},Y_{1;3}Y_{2;3}]=Y_{1;2}^{-1}Y_{2;3}^{-1}Y_{1;3}^{-1}Y_{1;2}Y_{1;3}Y_{2;3}=Y_{1;2}^{-2}$.
       Similarly, we have that
        \begin{eqnarray*}
	Y_{1;2}^{-1}Y_{3;2}^{-1}Y_{1;3}Y_{3;2}Y_{1;2}
	&=&
	Y_{1;2}^{-1}Y_{3;2}^{-1}s_1([\gamma_2])Y_{3;2}Y_{1;2}\\
	&=&
	Y_{1;2}^{-1}s_1([\gamma_1^{-1}\gamma_2^{-1}\gamma_1])Y_{1;2}\\
	&=&
	s_1([\gamma_2^{-1}])\\
	&=&
	Y_{1;3}^{-1}
       \end{eqnarray*}
       (see Figure~\ref{(3)}~(b)).
       Hence we have that $[Y_{1;3},Y_{3;2}Y_{1;2}]=Y_{1;3}^{-2}$.
       Similar to (1), we have that $Y_{1;2}^{-2}$ and $Y_{1;3}^{-2}$
       are BSCC maps of type $(1,2)$.
       \begin{figure}[htbp]
	\subfigure[]{\includegraphics[scale=0.5]{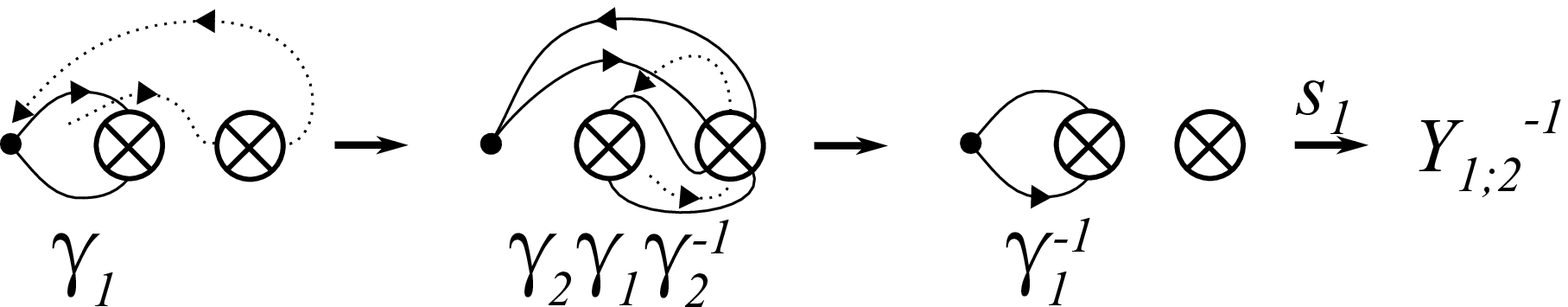}}

	\subfigure[]{\includegraphics[scale=0.5]{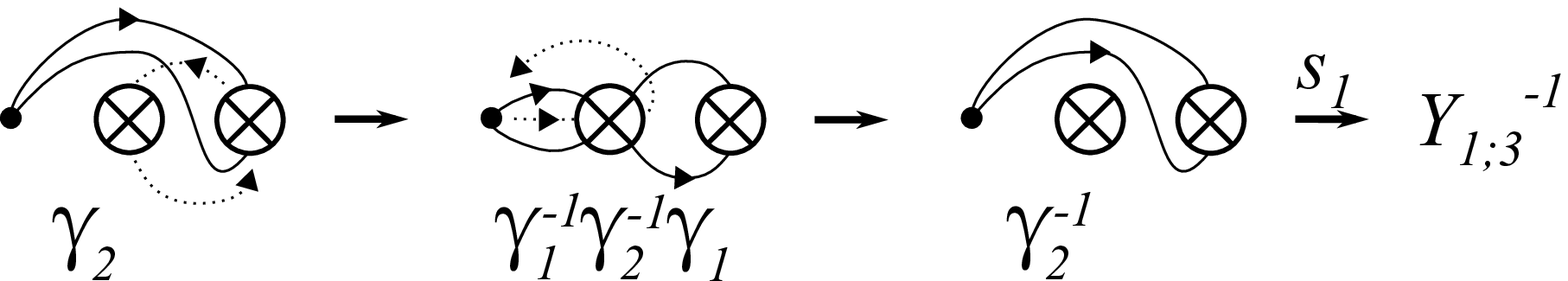}}
	\caption{}\label{(3)}
       \end{figure}
 \item In $\I(N_g)$, it is clear that $[Y_{1;2},Y_{3;4}]$ and
       $[Y_{2;3}^{-2}Y_{1;3}Y_{2;3}^2,Y_{2;4}]$ are equal to $1$.
       Note that
       $Y_{2;3}^{-2}Y_{1;3}Y_{2;3}^2=s_1([\gamma_1^{-2}\gamma_2\gamma_1^2])$
       (see Figure~\ref{(4)}).
       \begin{figure}[htbp]
	\includegraphics[scale=0.5]{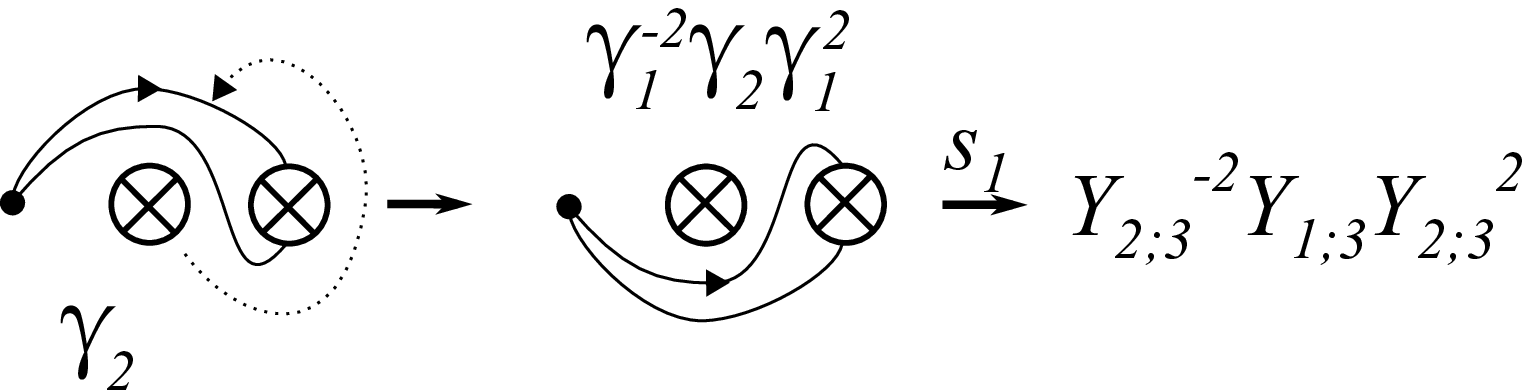}
	\caption{}\label{(4)}
       \end{figure}
 \item We have that
       $(Y_{1;2}Y_{1;3}Y_{1;4})^2=s_1([(\gamma_3\gamma_2\gamma_1)^2])=t_{c_1}t_{c_2}$,
       where $c_1$ and $c_2$ are simple closed curves as shown in
       Figure~\ref{(5)}.
       Note that $t_{c_1}$ is a BSCC map of type $(1,2)$ if $g\geq5$ and
       a BSCC map of type $(2,1)$ if $g=4$.
       Since $c_2$ bounds a ${\rm M\ddot{o}bius}$ band, $t_{c_2}$ is
       trivial.
       Hence we have that $(Y_{1;2}Y_{1;3}Y_{1;4})^2$ is a BSCC map of
       type $(1,2)$ if $g\geq5$ and a BSCC map of type $(2,1)$ if
       $g=4$.
       \begin{figure}[htbp]
	\includegraphics[scale=0.5]{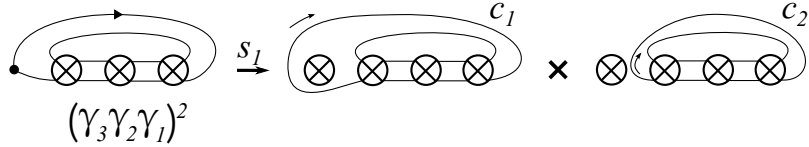}
	\caption{The black cross ``$\times$'' means the composition of
	$t_{c_1}$ with $t_{c_2}$.}\label{(5)}
       \end{figure}
 \item By the proof of Lemma~3.1 in \cite{s2}, we have
       $Y_{j;i}^{-1}Y_{i;j}=Y_{j;i}Y_{i;j}^{-1}=T_{i,j}^2$, for
       $1\leq{i<j}\leq{g}$.
       Hence we have that
       $(Y_{2;1}^{-1}Y_{1;2}Y_{3;2}Y_{2;3}^{-1}Y_{1;3}^{-1}Y_{3;1})^2=(T_{1,2}^2T_{2,3}^2T_{1,3}^{-2})^2$.
       \begin{lem}
	We have followings (see Figure~\ref{(6)L}).
	\begin{enumerate}
	 \item $T_{1,2}^2=L_1Y_{3;1}Y_{3;2}$, where
	       $L_1=(Y_{g;1}Y_{g;2}Y_{g;3}^2\cdots{Y_{g;g-1}^2})\cdots(Y_{4;1}Y_{4;2}Y_{4;3}^2)$.
	 \item $T_{2,3}^2=L_2Y_{1;2}Y_{1;3}$, where
	       $L_2=(Y_{g;1}^2Y_{g;2}Y_{g;3}Y_{g;4}^2\cdots{Y_{g;g-1}^2})\cdots(Y_{4;1}^2Y_{4;2}Y_{4;3})$.
	 \item $T_{1,3}^{-2}=L_3Y_{2;3}Y_{2;1}$, where
	       $L_3=(Y_{g;3}^{-1}Y_{g;1}Y_{g;2}^2\cdots{Y_{g;g-1}^2})\cdots(Y_{4;3}^{-1}Y_{4;1}Y_{4;2}^2Y_{4;3}^2)$.
	\end{enumerate}
	\end{lem}
       \begin{figure}[htbp]
	\includegraphics[scale=0.5]{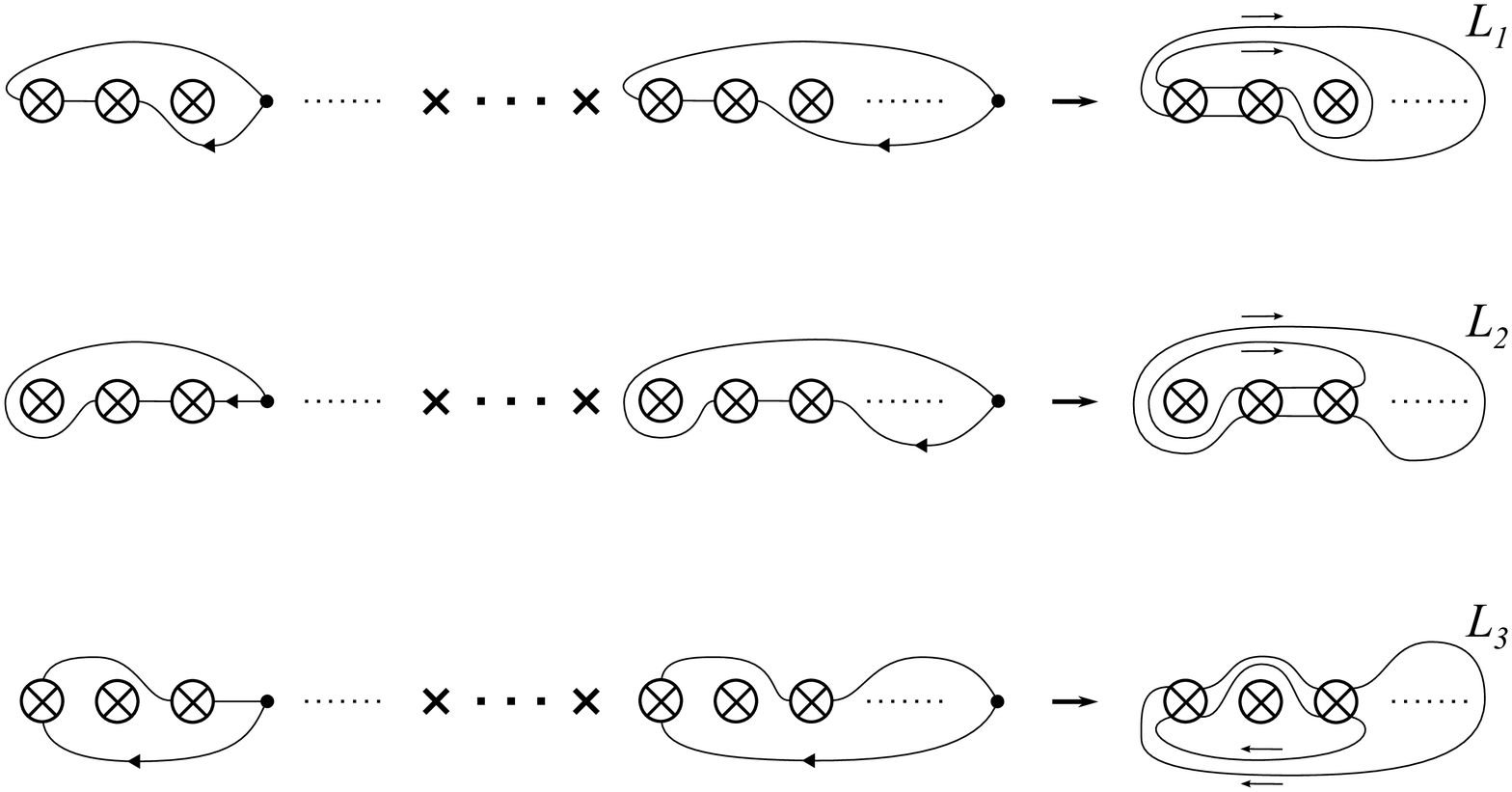}
	\caption{In this figure, and also Figures~\ref{(6)R132},
	\ref{(7)} and \ref{exam}, the black crosses ``$\times$'' mean
	compositions of the crosscap pushing maps.}\label{(6)L}
       \end{figure}
	We have
	\begin{eqnarray*}
	 &&Y_{3;2}^{-2}Y_{3;1}^{-1}(Y_{2;1}^{-1}Y_{1;2}Y_{3;2}Y_{2;3}^{-1}Y_{1;3}^{-1}Y_{3;1})^2Y_{3;1}Y_{3;2}^2\\
	 &=&Y_{3;2}^{-2}Y_{3;1}^{-1}\cdot{}L_1Y_{3;1}Y_{3;2}\cdot{}Y_{3;2}Y_{2;3}^{-1}\cdot{}L_3Y_{2;3}Y_{2;1}\\
	 &&\cdot{}Y_{2;1}^{-1}Y_{1;2}\cdot{}L_2Y_{1;2}Y_{1;3}\cdot{}Y_{1;3}^{-1}Y_{3;1}\cdot{}Y_{3;1}Y_{3;2}^2\\
	 &=&R_1R_3R_2Y_{1;2}^2Y_{3;1}^2Y_{3;2}^2,
	\end{eqnarray*}
	where $R_1=Y_{3;2}^{-2}Y_{3;1}^{-1}L_1Y_{3;1}Y_{3;2}^2$,
	$R_2=Y_{1;2}L_2Y_{1;2}^{-1}$ and $R_3=Y_{2;3}^{-1}L_3Y_{2;3}$
       (see Figure~\ref{(6)R}).
	Similar to (1), since $Y_{1;2}^2$, $Y_{3;1}^2$ and $Y_{3;2}^2$
	are BSCC maps of type $(1,2)$, it is suffice to show that
	$R_1R_3R_2$ is a product of BSCC maps.
       \begin{figure}[htbp]
	\includegraphics[scale=0.5]{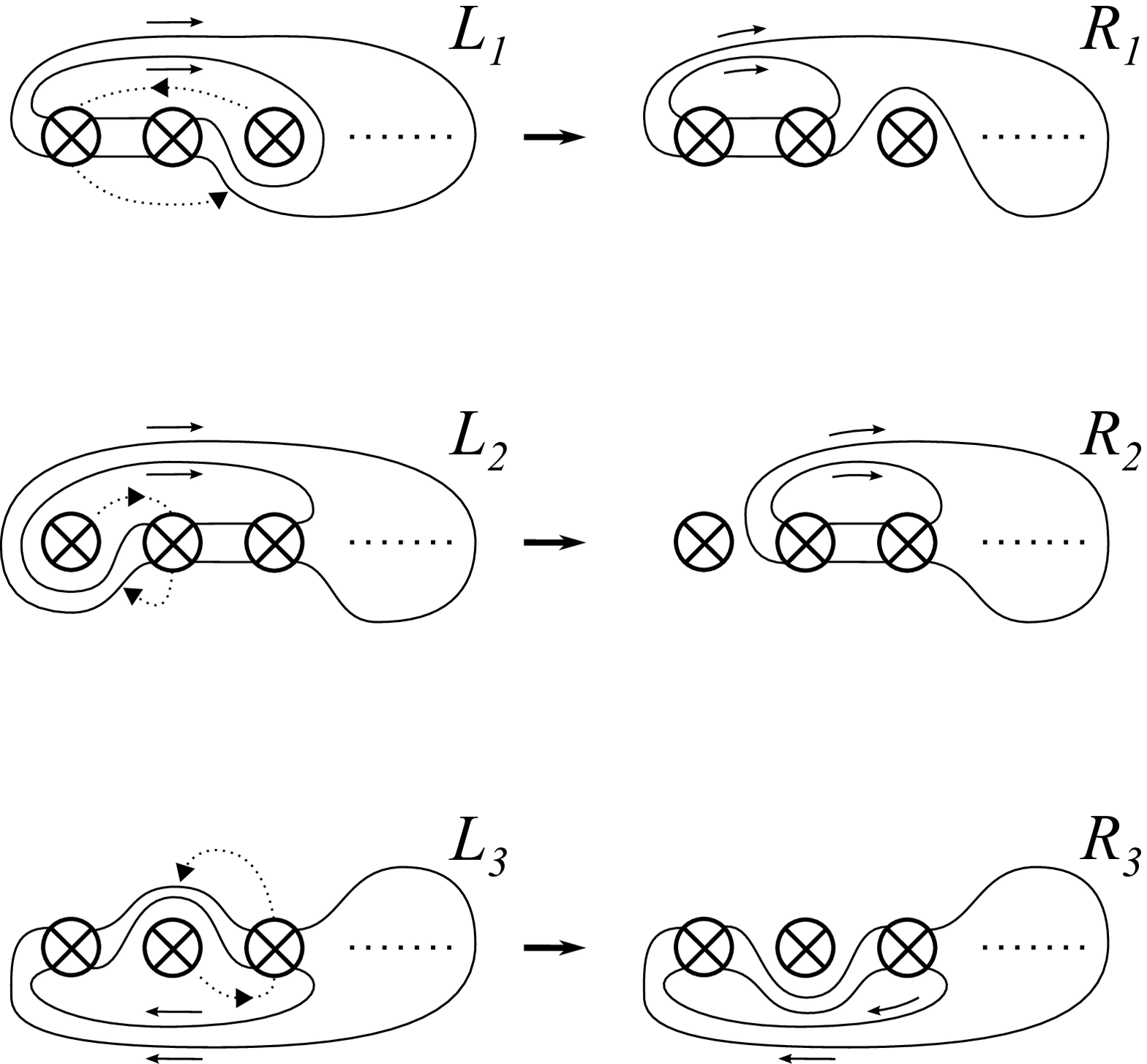}
	\caption{}\label{(6)R}
       \end{figure}
       Let $d_0$, $d_1$, $d_2$, $d_3$ and $d_4$ be simple closed curves
       as shown in Figure~\ref{(6)R132}.
       Note that $t_{d_0}$ is a BSCC map of type $(1,3)$ or $(1,g-3)$,
       and $t_{d_3}$ is a BSCC map of type $(2,1)$.
       In addition, since $d_4$ bounds a M\"obius band, we have
       $t_{d_4}=1$.
       Let $R_{32}=t_{d_1}t_{d_2}$.
       By the lantern relation, we have that $R_{32}$ is a product of
       $R_3R_2$ and $t_{d_0}$.
       Let $R_{132}=t_{d_3}t_{d_4}$.
       By the lantern relation, we have that $R_{132}$ is a product of
       $R_1R_{32}$ and $t_{d_0}$.
       Hence we have that $R_{132}$ is a product of $R_1R_3R_2$ and
       $t_{d_0}$.
       Therefore $R_1R_3R_2$ is a product of BSCC maps of type one and a
       BSCC map of type $(2,1)$.
       In particular, if $g=4$, since $t_{d_0}$ is trivial, we have that
       $R_1R_3R_2$ is a BSCC map of type $(2,1)$, and if $g\geq5$, by
       Lemma~\ref{lem}, we have that $R_1R_3R_2$ is a product of BSCC
       maps of type $(1,2)$.
       \begin{figure}[htbp]
	\includegraphics[scale=0.5]{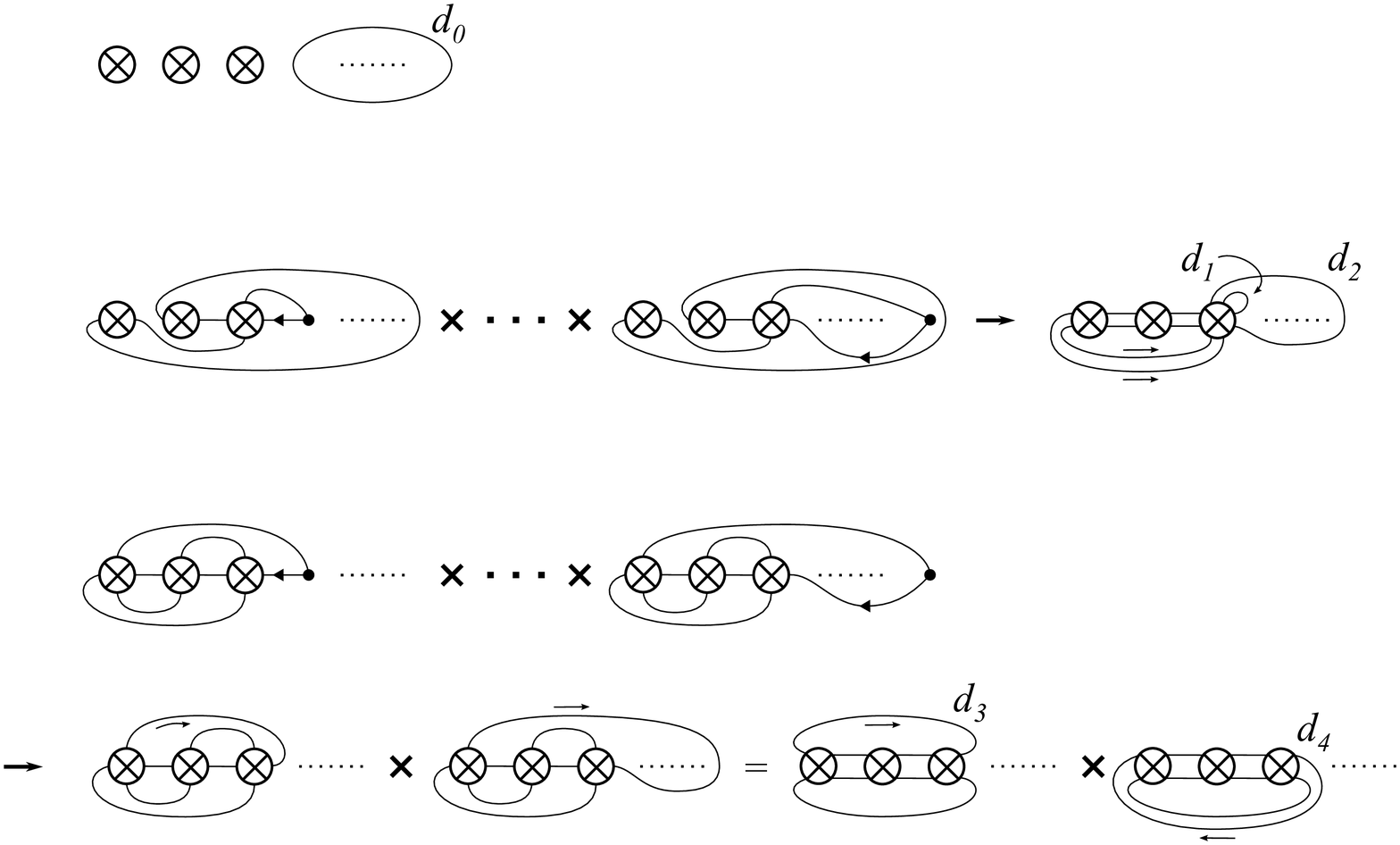}
	\caption{}\label{(6)R132}
       \end{figure}
 \item We have 
       \begin{eqnarray*}
	T_{1,2,3,4}^{-1}T_{1,2,3,4}'^{-1}
	&=&\prod_{5\leq{m}\leq{g}}s_m([\gamma_1^{-1}\gamma_2^{-1}\gamma_3^{-1}\gamma_4^{-1}\gamma_5^{-2}\gamma_6^{-2}\cdots\gamma_{m-1}^{-2}])\\
	&=&\prod_{5\leq{m}\leq{g}}Y_{m;m-1}^{-2}\cdots{Y_{m;6}^{-2}}Y_{m;5}^{-2}Y_{m;4}^{-1}Y_{m;3}^{-1}Y_{m;2}^{-1}Y_{m;1}^{-1}.
       \end{eqnarray*}
       (see Figure~\ref{(7)}).
       Hence we have
       $$T_{1,2,3,4}^2(\prod_{5\leq{m}\leq{g}}Y_{m;m-1}^{-2}\cdots{Y_{m;6}^{-2}}Y_{m;5}^{-2}Y_{m;4}^{-1}Y_{m;3}^{-1}Y_{m;2}^{-1}Y_{m;1}^{-1})=T_{1,2,3,4}T_{1,2,3,4}'^{-1}.$$
       Thus it is a BP map of type $(1,1)$.
       \begin{figure}[htbp]
	\includegraphics[scale=0.5]{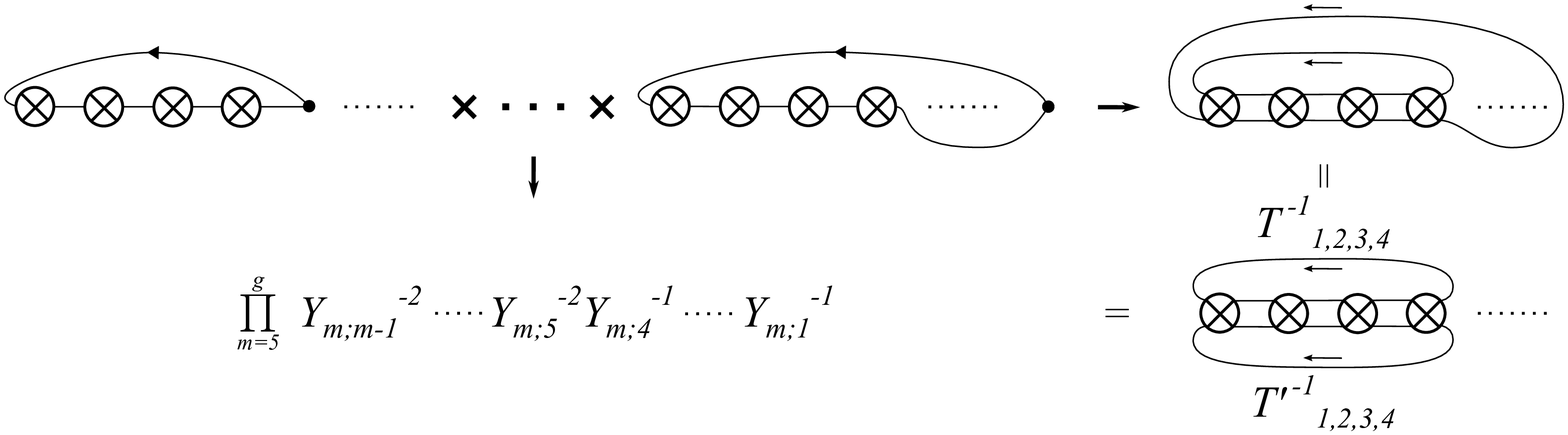}
	\caption{}\label{(7)}
       \end{figure}
\end{enumerate}
Thus we complete the proof.
\end{proof}

\appendix
\section{}\label{appendix}

In this appendix, we show the equation~(\ref{append(I)}) in
Section~\ref{basic} and the equation~(\ref{append(II)}) in
Section~\ref{norgen}.

For (\ref{append(I)}), it suffices to show the following lemma.

\begin{lem}
For $1\leq{}i_1<i_2<\cdots<i_k\leq{g}$, let $c_{i_1,\dots,i_k}$ be a
 simple closed curve on $N_g$ as shown in Figure~\ref{loopI3}.
Then for $2\leq{h}\leq{g}$ we have
$$t_{c_{1,\dots,h}}=(t_{c_{1,2}}\cdots{t_{c_{1,h-1}}t_{c_{1,h}}})\cdots(t_{c_{h-2,h-1}}t_{c_{h-2,h}})(t_{c_{h-1,h}}).$$
\begin{figure}[htbp]
\includegraphics[scale=0.5]{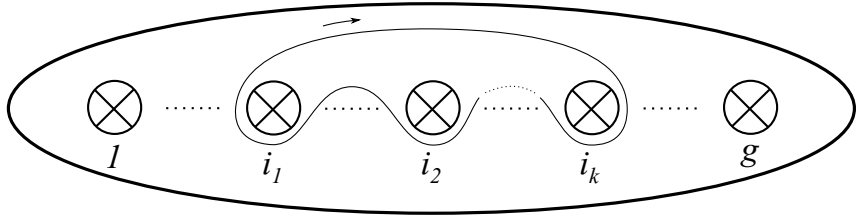}\hspace{0.5cm}
\caption{The curve $c_{i_1,\dots,i_k}$ on $N_g$.}\label{loopI3}
\end{figure}
\end{lem}

\begin{proof}
We first note that $t_{c_{i,j}}=Y_{i;j}^2$.
We see
\begin{eqnarray*}
t_{c_{1,\dots,h}}
&=&(t_{c_{1,\dots,h}}t_{c_{2,\dots,h}}^{-1})\cdots(t_{c_{h-2,h-1,h}}t_{c_{h-1,h}}^{-1})(t_{c_{h-1,h}}t_{c_h}^{-1})\\
&=&(s_1([\gamma_{h-1}^2\cdots\gamma_1^2]))\cdots(s_{h-2}([\gamma_{h-1}^2\gamma_{h-2}^2]))(s_{h-1}([\gamma_{h-1}^2]))\\
&=&(Y_{1;2}^2\cdots{}Y_{1;h-1}^2Y_{1;h}^2)\cdots(Y_{h-2;h-1}^2Y_{h-2;h}^2)(Y_{h-1;h}^2)\\
&=&(t_{c_{1,2}}\cdots{t_{c_{1,h-1}}t_{c_{1,h}}})\cdots(t_{c_{h-2,h-1}}t_{c_{h-2,h}})(t_{c_{h-1,h}}).
\end{eqnarray*}
Thus we obtain the claim.
\end{proof}

\begin{exam}
For $1\leq{i}\leq{g-1}$ we have
\begin{eqnarray*}
Y_{g;i}&=&
(Y_{1;2}^2\cdots{}Y_{1;g-1}^2Y_{1;i}^{-1}Y_{1;g})\cdots(Y_{i-1;i}^2\cdots{}Y_{i-1;g-1}^2Y_{i-1;i}^{-1}Y_{i-1;g})\\
&&\cdot(Y_{i+1;i+2}^2\cdots{}Y_{i+1;g-1}^2Y_{i+1;i}^{-1}Y_{i+1;g}Y_{i+1;i}^2)\cdots(Y_{g-2;g-1}^2Y_{g-2;i}^{-1}Y_{g-2;g}Y_{g-2;i}^2)\\
&&\cdot(Y_{g-1;i}^{-1}Y_{g-1;g}Y_{g-1;i}^2)Y_{i;g}.
\end{eqnarray*}
\end{exam}

\begin{proof}
Note that $Y_{g;i}Y_{i;g}^{-1}=T_{i,g}^2$.
We see
\begin{eqnarray*}
T_{i,g}^2
&=&\prod_{1\leq{m}\leq{i-1}}s_m([\gamma_{g-1}\gamma_{i-1}^{-1}\gamma_{g-2}^2\cdots\gamma_m^2])\prod_{i+1\leq{m}\leq{g-2}}s_m([\gamma_i^2\gamma_{g-1}\gamma_i^{-1}\gamma_{g-2}^2\cdots\gamma_m^2])\\
&&\cdot{}s_{g-1}([\gamma_i^2\gamma_{g-1}\gamma_i^{-1}])\\
&=&\prod_{1\leq{m}\leq{i-1}}(Y_{m;m+1}^2\cdots{}Y_{m;g-1}^2Y_{m;i}^{-1}Y_{m;g})\prod_{i+1\leq{m}\leq{g-2}}(Y_{m;m+1}^2\cdots{}Y_{m;g-1}^2Y_{m;i}^{-1}Y_{m;g}Y_{m;i}^2)\\
&&\cdot(Y_{g-1;i}^{-1}Y_{g-1;g}Y_{g-1;i}^2)
\end{eqnarray*}
(see Figure~\ref{exam}).
Thus we obtain the claim.
\begin{figure}[htbp]
\includegraphics[scale=0.5]{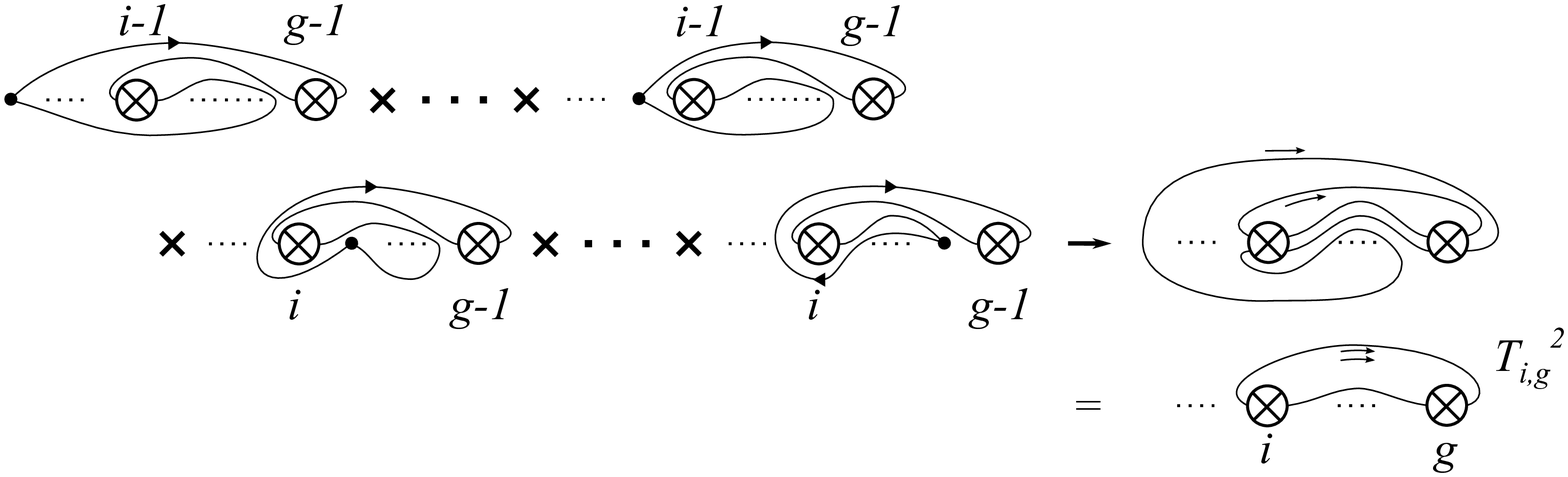}\hspace{0.5cm}
\caption{}\label{exam}
\end{figure}
\end{proof}
Thus we obtain the equation~(\ref{append(II)}).



\end{document}